\documentclass[a4paper,10pt]{scrartcl}
\usepackage{amsmath,amssymb,amsthm}
\usepackage{graphicx}
\usepackage{color}
\usepackage{chemfig} 
\usepackage{textcomp}
\usepackage{float}
\usepackage{hyperref}
\usepackage[caption = false]{subfig}
\usepackage{cleveref}
\usepackage{tikz}

\newtheorem{theorem}{Theorem}
\newtheorem{corollary}[theorem]{Corollary}
\newtheorem{lemma}[theorem]{Lemma}
\newtheorem{proposition}[theorem]{Proposition}

\theoremstyle{definition}

\newtheorem{remark}[theorem]{Remark}

\newcommand{\R}{\mathbf{R}}
\newcommand{\N}{\mathbf{N}}

\newcommand{\hd}{\mathcal{H}}

\DeclareMathOperator{\dive}{div}

\renewcommand{\d}{{\mathrm d}}
\newcommand{\dist}{{\mathrm{dist}}}

\newcommand{\meas}{{\mathcal{M}}}

\renewcommand{\Re}{{\mathrm{Re}}}
\newcommand{\Pe}{{\mathrm{Pe}}}
\newcommand{\spt}{{\mathrm{spt}}}

\newlength{\dhatheight}

\title{A mathematical model for cell polarization in zebrafish primordial germ cells}
\author{Carolin Dirks \and Paul Striewski \and Benedikt Wirth \and Anne Aalto \and Adan Olguin-Olguin \and Erez Raz}
\date{}

\begin{document}
\maketitle

\begin{abstract}
Blebs are cell protrusions generated  by local membrane--cortex detachments followed by expansion of the plasma membrane. Blebs are formed by some migrating cells, for example primordial germ cells of the zebrafish. 
While blebs occur randomly at each part of the membrane in unpolarized cells, a polarization process guarantees the occurrence of blebs at a preferential site and thereby facilitates migration 
towards a specified direction. Little is known about the factors involved in development and maintenance of a polarized state, yet recent studies revealed the influence of an intracellular flow and the stabilizing 
role of the membrane-cortex linker molecule Ezrin. Based on this information, we develop and analyse a coupled bulk-surface model describing a potential cellular mechanism by which a bleb could be induced at a 
controlled site. The model rests upon intracellular Darcy flow and a diffusion-advection-reaction system, describing the temporal evolution from an unpolarized to a stable polarized Ezrin distribution. 
We prove the well-posedness of the mathematical model and show that simulations qualitatively correspond to experimental observations,
suggesting that indeed the interaction of an intracellular flow with membrane proteins can be the cause of the cell polarization.
\end{abstract}


\section{Introduction}

Several recent studies investigated the directional cell migration process via local membrane protrusions, so-called blebs. While the mechanisms of the actual 
bleb formation are quite well understood, the process of cell polarization leading to a stable \textit{directional} blebbing remains still unexplained. In 
some recent works (such as \cite[Paluch, Raz, 2013]{Ref:PaluchRaz2013}, \cite[Fritzsche, Thorogate et al., 2014]{Ref:FritzscheEtAl2014}), researchers suggested 
the role of the membrane-cortex linker Ezrin in inhibiting the probability for bleb formation in regions with a high Ezrin concentration.
In addition, a directed intracellular flow has been observed during cell polarization that seems to be related to the occurrence of so-called actin brushes, filamentous actin structures 
forming at the front side of the cell \cite[Kardash, Reichmann-Fried, 2010]{Ref:KardashEtAl2010}.
In this article, we take up these observations and hypothesize that shear stresses induced by the intracellular flow may lead to a local destabilization of the Ezrin linkages 
between membrane and cortex, resulting in a redistribution of membranous Ezrin and bleb formation.
This hypothesis is tested using a mathematical model for the time interval between actin brush formation and the onset of blebbing.
The model incorporates an intracellular flow driven by actin brushes and a description of the 
flow-controlled membranous Ezrin concentration including turnover rates from active (membrane-bound) and inactive (cytosolic) Ezrin. 
The experimentally observed Ezrin depletion in the front and accumulation in the back of the cell can be reproduced by the model.
Thereby our model positively answers the question whether there could be a mechanical basis for Ezrin polarization, in our case an actin-induced flow.


This work is organized as follows. We start with providing a brief overview of the biological context and the related work, and we introduce the notation used throughout this article. In 
\cref{sec:Model}, we describe our model used to simulate the temporal behaviour of the active Ezrin. The corresponding model analysis is presented in 
\cref{sec:Analysis}, where we prove well-posedness of the surface equation. Finally, we describe the numerical treatment of the coupled bulk-surface 
equation system and compare simulation results to experiments in \cref{sec:Experiments}.

\subsection{Biological setting} 

The process of directional cell migration is an important and extensively studied mechanism in early embryonic development. A widely used model for in 
vivo studies are primordial germ cells (PGCs). These cells are specified within the embryo and have to travel a certain distance to reach their 
destination, namely the site where the gonad develops \cite[Doitsidou, Reichmann-Fried et al., 2002]{Ref:DoitsidouEtAl2002}. This migration process is performed 
via blebs, local detachments of the cell membrane from the cortex which move the cell to a certain direction specified by a chemical gradient. Little 
is known about the signaling process within the cell in the time interval between the arrival of the chemical signal and the actual directed movement, in which 
the cell changes from an unpolarized to a polarized state. However, several factors have been shown to play a role in the polarization process \cite[Paluch, 
Raz, 2013]{Ref:PaluchRaz2013}. 

Blebbing is produced by an increase in the intracellular pressure coupled to detachment of the cell membrane from the cell cortex. 
While migrating, PGCs go through two different phases, named ``run'' and ``tumble''. 
During the ``tumble'' state, PGCs are apolar and blebs are formed at random sites around the cell perimeter. 
When the PGCs are in the ``run'' state the cells are polarized such that blebs form predominantly in one direction which is defined as the leading edge \cite[Kardash, Reichmann-Fried, 2010]{Ref:KardashEtAl2010}, \cite[Paksa, Raz, 2015]{Ref:PaksaRaz2015}.


Although the entire process of PGC polarization has not yet been fully understood, some factors have been identified to wield a strong influence. In 
the polarized state, a preferential polymerization of filamentous actin structures, so-called actin brushes, at the front edge of cell was reported, whereas 
such structures are absent in unpolarized cells \cite[Kardash, Reichmann-Fried, 2010]{Ref:KardashEtAl2010}. The actin brushes are considered to be responsible for a recruitment of myosin, that leads to an increase of the contractility, favouring the corresponding side of the cell as the leading edge \cite[Paluch, Raz, 2013]{Ref:PaluchRaz2013}.
The accumulation of actin brushes is furthermore assumed to be correlated with a flow of cytoplasm 
towards the expanding bleb on the one hand and a strong retrograde flow of cortical actin on the other \cite[Kardash, Reichmann-Fried, 
2010]{Ref:KardashEtAl2010}, \cite[Reig, Pulgar et al., 2014]{Ref:ReigEtAl2014}. Moreover, a frequently reported feature of polarized blebbing cells is a local 
decrease of the membrane-cortex attachment at the front edge in combination with an increase at the back. Hence, a negative correlation between the propensity 
for blebbing and the stability of membrane-cortex attachment is assumed. A presumable candidate for regulating the membrane-cortex attachment is the linker 
molecule Erzin \cite[Paluch, Raz, 2013]{Ref:PaluchRaz2013}. Experiments have shown that in polarized cells, Ezrin accumulates at the back \cite[Lorentzen, Bamber et al.,2011]{Lorentzen1256}. 
Besides, the linker molecule is able to switch between an active and an inactive state. During its active form, 
it links the cell cortex to the membrane via two binding terminals (the membrane-binding N-terminal and the actin-binding C-terminal), whereas in its inactive 
form, these terminals interact with each other causing the molecule to diffuse within the cytoplasm. 
Ezrin constantly keeps turning from one state to the other, resulting in a frequent change between binding to and detaching from the membrane \cite[Fritzsche, 
Thorogate et al., 2014]{Ref:FritzscheEtAl2014}, \cite[Br\"uckner, Pietuch et al., 2015]{Ref:BruecknerEtAl2015}. 

To get a better understanding of the intracellular events involved in the polarization processes, we develop a mathematical model expressing 
the interaction of different factors which are known to play a role in the emergence and maintenance of a polarized state. The model focuses on the role and 
regulation of the active and inactive Ezrin concentration, including the influence of the cytoplasmic flow driven by localised actin-myosin contraction. In 
particular, we present a potential model for the binding and unbinding dynamics along the cell membrane by incorporating the reported information together with 
reviewing some additional hypotheses.

\subsection{Related work}
A variety of models for cell polarization have been proposed, many of them based on reaction-diffusion equations, suggesting that diffusive instabilities are 
involved in the process of cell polarization \cite[Levine et al. 2006]{Ref:Levine}, \cite[Onsum, Rao, 2007]{Ref:OnsumRao}, \cite[R\"atz, R\"oger, 2012]{Ref:RaetzRoeger2012}, \cite[R\"atz, R\"oger, 2014]{Ref:RaetzRoeger2014}.
Cell polarization induced by active transport of polarization markers was for example studied by \cite[Hawkins, Bénichou et al., 2009]{Ref:Hawkins}, 
\cite[Calvez, Hwakins et al., 2012]{Ref:Calvez}. In either article, the presented models account for active transport of polarization markers along the cytoskeleton. 
\cite[Hausberg, R\"oger, 2018]{hausberg2018well} described the activity of GTPases by a system of three coupled bulk-surface advection-reaction-diffusion equations. 
The system models the interconversion of active and inactive GTPase, lateral drift and diffusion of molecules along the membrane and also the diffusion of inactive molecules into the cytosol. In contrast to our approach, 
Hausberg and R\"oger suggest flow-independent reaction terms and assume the geometry of the cell to be more regular than we do.

\cite[Garcke, Kampmann et al., 2015]{Ref:GarckeKampmann2015} proposed a model for lipid raft formation in cellular membranes and their interaction with 
intracellular cholesterol. Although not directly linked to cell polarization, their model comprises phase separation and interaction energies, which are similar 
to those presented in this article. In their work, a diffusion equation is to account for the intracellular diffusion of cholesterol, whereas a Cahn--Hilliard 
equation coupled with a reaction-diffusion-type equation models the formation of rafts and cholesterol binding and unbinding dynamics on the membrane.

In a recent work, \cite[Burger, Pietschmann et al., 2019]{Ref:BurgerPietschmann2019} presented a model for bleb formation based on a combination of intracellular flow and interactions of linker molecules. 
Their model includes a variable domain, where a bleb is initiated by a certain threshold distance between membrane and cortex. The flow inside and outside of the cell domain is modelled via Stokes equations, 
interactions of active and inactive linker molecules are described by a reaction-diffusion system. 
Although the main model ingredients are similar, the model details (such as the type of flow or our resulting phase field equation) as well as the mathematical analysis differ,
where in both cases we aim to impose a minimum amount of extra assumptions.
Additionally, the numerical simulations in \cite[Burger, Pietschmann et al., 2019]{Ref:BurgerPietschmann2019} target the membrane-cortex disruption,
while the aim of our approach is to investigate the effect of different model parameters on the global distribution of the linker molecules
and to compare the results directly to biological experiments. 

Outside of the scope of this work are models that describe the process of the actual bleb formation, as presented for example in \cite[Young, Mitran, 
2010]{Ref:YoungMitran} or \cite[Strychalsky, Guy, 2012]{Ref:StrychalskiGuy}.

\subsection{Notation}

Throughout the article, we assume that $\Omega$ is an open and convex bounded domain in $\R^n$ with boundary $\Gamma = \partial \Omega$, modelling a cell and its 
membrane. For a function $g$ defined on $(0,T)\times\Omega$, we denote the temporal derivative and the gradient with respect to the spatial variable by 
$\partial_t g$ and $\nabla g$ respectively. The symbols $\nabla_{\Gamma}$, $\Delta_{\Gamma}$ and $\mathrm{div}_{\Gamma}$ are used to denote the surface 
gradient, Laplace-Beltrami operator and surface divergence on $\Gamma$. Additionally, we make use of the following conventions (the units refer to a three-dimensional cell):  

\begin{center}
\begin{tabular}{ll}
$w:(0,T)\times\Omega\rightarrow \R^n$ & intracellular flow [$\mu$m/s] \\
$p:(0,T)\times\Omega\rightarrow \R$ & intracellular pressure [Pa] \\
$f:(0,T)\times\Omega\rightarrow \R^n$ & flow-inducing force per unit volume [N/m$^3$] \\
$u:(0,T)\times\Gamma\rightarrow [0,1]$ & density of active Ezrin [mol/$\mu$m$^2$] \\
$v:(0,T)\times\Omega\rightarrow [0,1]$ & density of inactive Ezrin [mol/$\mu$m$^3$] \\
$\kappa:\Omega\rightarrow\R$ & local permeability of cytoplasmatic matrix [m$^2$] \\
$\rho$ & density of cytoplasm [g/cm$^3$] \\
$\lambda$ & kinematic viscosity of cytoplasm [cm$^2$/s] \\
$\nu$ & diffusion coefficient in cytoplasm [$\mu$m$^2$/s]
\end{tabular}
\end{center}


\section{Model derivation} \label{sec:Model}

In this chapter, we present and describe a coupled bulk-surface model for the contribution of the linker-molecule Ezrin to the polarized state of a blebbing 
cell. The model covers the time interval in between the formation of the actin brushes, creating an intracellular flow, and the point where the cell arrives in a stable polarized state. The key points of the model can be summarized as follows: We assume a circular intracellular 
flow driven by actin-myosin motors, which creates a shear stress along the cell membrane. This stress destabilizes membrane-bound Ezrin, which as 
a consequence undergoes a retrograde flow to the back side of the cell. In addition, a nonlinear self-enhancing effect is modelled by a stronger affinity of inactive Ezrin to bind in regions with a high 
active Ezrin concentration and vice versa. The system variables of our model are 

\begin{center}
\begin{tabular}{cl}
$w$ & intracellular flow \\
$p$ & intracellular pressure \\
$u$ & density of active (membrane-bound) Ezrin \\
$v$ & density of inactive (unbound) Ezrin.
\end{tabular} 
\end{center}

\subsection{Model description}

The model presented in this section aims at incorporating all biological information assumed to play a role in bleb formation, while trying to be as simple as 
possible on the other side. In detail it is based on the following observations:

\begin{enumerate}
	
\item Very recent experiments \cite[Olguin-Olguin, Aalto et al., 2019]{Ref:OlguinEtAl2019} suggest that the membrane-cortex linker Ezrin becomes localised to the cell back where it functions in inhibiting bleb formation.

        
\item Experiments observed high cytosolic diffusion rates of inactive Ezrin in the range of 30\,$\mu$m$^2$/s \cite[Coscoy et al.]{Ref:Coscoy}. We show (see 
\cref{sec:Nondimensionalisation}) that as a consequence, on the time scales of interest cytosolic concentrations of inactive Ezrin are therefore likely to be spatially constant throughout the cell. 
On the other hand, \cite[Fritzsche et al.]{Ref:FritzscheEtAl2014} reported a slow diffusion of membranous Ezrin of around 0.003\,$\mu$m$^2$/s, caused by differences between the 
binding characteristics of its N- and C-terminal.

\item Experiments suggest the existence of an intracellular flow \cite[Goudarzi, Tarbashevich et al., 2017]{goudarzi2017bleb}, presumably mediated by myosin 
motors in the actin brushes, resulting in a fountain-like motion pattern of cytoplasmic particles (also modelled in \cite[Strychalsky, Guy, 2012]{Ref:StrychalskiGuy}, for instance).
We hypothesize that beyond mere 
passive transport of Ezrin molecules, hydrodynamic shear stresses might lead to alteration of the binding and unbinding dynamics of membrane-cortex linkers, 
especially to an increased dissociation in regions of high stress.

\item As different studies revealed \cite[Berryman et al]{Ref:Berryman}, \cite[Gautreau et al.]{Ref:Gautreau}, inactive Ezrin can form oligomers in the 
cytoplasm and at the plasma membrane, and the article \cite[Herrig et al.]{Ref:Herrig} suggests that Ezrin binds cooperatively to membrane regions with a high 
concentration of certain types of phospholipids (PIP(2)).
Thus, the binding of Ezrin between cortex and membrane might be reinforced by higher active Ezrin concentrations.
On the other hand, a lack of membrane-cytoskeleton linkers could result in delamination of the membrane from 
the cortex, thereby preventing the binding of Ezrin to the cortex-membrane complex so that low active Ezrin concentrations may have an inhibitory effect.
Finally, a certain maximum density of linker molecules between cortex and membrane cannot be exceeded.
Altogether this suggests a nonlinear influence of the active Ezrin concentration on the Ezrin binding dynamics.
\end{enumerate}

A generic primordial germ cell (PGC) is expressed by a time-independent domain $\Omega \subset \R^n$ with a smooth boundary $\Gamma = \partial 
\Omega$ (which is temporally static throughout our work since we are only interested in the cell behaviour \emph{before} bleb formation). To simulate a realistic environment, one would choose $n=3$, while for the sake of simplicity, 
we set $n=2$ in our simulations, since no qualitative difference is expected in the behaviour between a two- or a three-dimensional cell. The 
density of active Ezrin (in mol/$\mu$m$^2$) is then described by a function $u:(0,T)\times\Gamma \to \R$, the density of inactive Ezrin (in mol/$\mu$m$^3$) as 
$v:(0,T)\times\Omega\rightarrow\R$.

As the cytoplasm can be described as a poroelastic material (see e.\,g.\ \cite[Moeendarbary et al.]{Ref:Moeendarbary}), cytoskeleton and organelles behaving like 
elastic solids and the cytosol like a fluid, the cytoplasmic flow $w$ is described by the incompressible Brinkman--Navier--Stokes equation 
\begin{align}
\rho\partial_t w + \rho(w\cdot\nabla)w + \frac{\rho\lambda}{\kappa}w &= -\nabla p + \rho\lambda\Delta w + f &&\text{ in } \Omega, \label{Eq:BNSEquation} \\
\mathrm{div}\, w &= 0\, ,&&\text{ in } \Omega\label{Eq:incompressibility}
\end{align}
with no-outflow boundary conditions. Here, $\kappa$ denotes the (spatially varying) permeability, $p$ 
denotes the intracellular pressure and $f$ a body force induced by actin-myosin contraction in the actin brushes. The parameter $\lambda$ represents the kinematic viscosity and 
$\rho$ the density of the cytoplasm.

The density $u$ of active Ezrin on $\Gamma$ is obtained as a solution of the reaction-advection-diffusion equation
\begin{align}
 \partial_t u &= -\underbrace{\mathrm{div}_{\Gamma}(u\, w)}_{\mathrm{advection}} + \underbrace{\vphantom{()}\nu\, \Delta_{\Gamma} u}_{\mathrm{diffusion}} - 
\underbrace{D(w,u,v)}_{\mathrm{desorption}} + \underbrace{A(w,u,v)}_{\mathrm{adsorption}} & \text{in }\ \Gamma\, . \label{Eq:SurfaceEquation}
\end{align}
The two functions $A,D:\R^n\times\R\times\R \to \R^{\geq 0}$ describe the adsorption and desorption kinetics of membranous Ezrin, which have been studied for instance in 
\cite[Bosk et al.]{Ref:Bosk}, but whose exact form is not known. In order to derive reasonable rate descriptions, we model the turnover between the two concentrations via 
classical reaction kinetics. To this end, denote by $u_{\max}$ the theoretical concentration of Ezrin binding sites on the cell membrane-cortex complex and by $b=u_{\max}-u$ the concentration of free sites.
Due to potential local membrane-cortex detachments, it might happen that some binding sites are actually not available for inactive Ezrin since the distance between 
membrane and cortex is too large. Hence, we denote by $\tilde{b}\leq b$ the binding site concentration where new Ezrin is actually allowed to bind. Since a locally 
higher active Ezrin concentration is associated with a smaller distance between membrane and cortex, there is a (at least to first order) linear relation between $b$ and $\tilde{b}$ via
\begin{equation*}
\tilde{b} = (\tfrac{u}{u_{\max}})^\alpha b
\end{equation*}
for some exponent $\alpha>0$, which for lack of experimental data we will choose as $\alpha=1$ (as in the case of membrane-binding experiments for the small Rho GTPase Cdc42 in \cite[Goryachev, 
Pokhilko 2008]{Ref:GoryachevPokhilko2008}; the choice does not affect the qualitative model behaviour).
The binding rates follow the classical reaction kinetics
\begin{center}
\schemestart 
inactive Ezrin $v$ + available binding-sites $\tilde{b}$ \arrow{<=>[$k_1$][$k_2(w)$]}active Ezrin $u$,
\schemestop
\end{center}
where for simplicity we assume a constant adsorption rate $k_1$ and a flow-dependent desorption rate $k_2(w)$. By the law of mass action this implies 
\begin{equation*}
\frac{du}{dt} = k_1 v \tilde{b} - k_2(w)u = k_1(\tfrac{u}{u_{\max}})^\alpha bv - k_2(w)u = k_1(\tfrac{u}{u_{\max}})^\alpha(u_{\max}-u)v - k_2(w)u\,, 
\end{equation*}
where in the last step we used $b = u_{\max}-u$.
Thus, with $\gamma=k_1/u_{\max}^\alpha$ we set
\begin{equation*}
A(w,u,v) = a(u,v) = \begin{cases} \gamma u^\alpha(u_{\max}-u)v & \text{ if } u\leq u_{\max}, \\ 0 & \text{ otherwise.} \end{cases}
\end{equation*}
The desorption rate $k_2(w)$ on the other hand is taken to depend affinely (which can be thought of as the first order expansion) on the tangential cytoplasm flow velocity along the membrane,
modelling a shear-stress induced Ezrin destabilization.
For Ezrin concentrations above the saturation concentration $u_{\max}$, which could in principle occur within the membrane due to passive transport by the flow,
we assume a strong nonlinear desorption rate increase with exponent $\zeta>1$ (the exact form is not expected to have any qualitative effect).
In summary, we set
\begin{align*}
D(w,u,v) = d(w,u) = (\beta_1|w|+\beta_2)\begin{cases}  u & \text{ if } u\leq u_{\max}, \\ \frac{u^\zeta+(\zeta-1)u_{\max}^\zeta}{\zeta u_{\max}^{\zeta-1}}& \text{otherwise,} \end{cases}
\end{align*}
where the parameter $\beta_1> 0$ controls the influence of hydrodynamical contributions in Ezrin dissociation. Thus, $D$ increases with a stronger flow as well as with an increasing concentration of active Ezrin.

The cytosol concentration $v:\Omega\to\R$ of inactive Ezrin is governed by an advection-diffusion equation,
\begin{equation}\label{Eq:pdeV}
\partial_t v = -\mathrm{div}(vw) + \mu\Delta v \qquad \text{ in } \Omega,
\end{equation}
where Ezrin production or decay is negligible on the time-scale considered.
The boundary conditions are dictated by the turnover rates between active and inactive Ezrin,
\begin{equation}\label{Eq:BCV}
\mu\nabla v\cdot n = D(w,u,v)-A(w,u,v)= d(w,u) - a(u,v) \qquad \text{ on } \Gamma.
\end{equation}

Summarizing, the full system of equations for the variables $w,p,u,v$ with their initial and boundary conditions reads 
\begin{align*}
\rho\partial_t w&= - \rho(w\cdot\nabla)w - \frac{\rho\lambda}{\kappa}w -\nabla p + \rho\lambda\Delta w + f &&\text{ on } (0,T)\times\Omega, \\
0&=\mathrm{div}\, w &&\text{ on } (0,T)\times\Omega, \\
\partial_t u &= -\mathrm{div}_{\Gamma}(u\, w) + \nu\, \Delta_{\Gamma} u - d(w,u) + a(u,v) &&\text{ on } (0,T)\times\Gamma, \\
\partial_t v &= -\mathrm{div}(vw) + \mu\Delta v &&\text{ on } (0,T)\times\Omega, \\
w(0,\cdot) &= w_0 &&\text{ on } \Omega, \\
u(0,\cdot) &= u_0 &&\text{ on } \Gamma, \\
v(0,\cdot) &= v_0 &&\text{ on } \Omega, \\
w\cdot n &= 0 &&\text{ on } (0,T)\times\Gamma, \\
\mu\nabla v\cdot n &= d(w,u) - a(u,v) &&\text{ on } (0,T)\times\Gamma.
\end{align*}
Note that the (third) equation for $u$ can be seen as an Allen--Cahn equation with nonlinearity $W_{w,v}'(u)=d(w,u)-a(u,v)$, or equivalently with potential 
\begin{equation*}
W_{w,v}(u) = \begin{cases} \frac{\beta_1|w|+\beta_2}{2}(u^2-u_{\max}^2) -\gamma v\left(\frac{u_{\max}u^{\alpha+1}}{\alpha+1}-\frac{u^{\alpha+2}}{\alpha+2}-\frac{u_{\max}^{\alpha+2}}{(\alpha+1)(\alpha+2)}\right) & \text{ if } u\leq u_{\max},\\
\left(\beta_1|w|+\beta_2\right)\left(\frac{u^{\zeta+1}}{\zeta(\zeta+1)u_{\max}^{\zeta-1}} + \frac{(\zeta-1)u_{\max}}{\zeta}u-\frac{\zeta u_{\max}^2}{\zeta+1}\right) & \text{ otherwise,} \end{cases}
\end{equation*}
and additional advection term.
Illustrative sketches of the form and behaviour of $W_{w,v}$ are provided in \cref{fig:doubleWell1,fig:doubleWell2}.
Allen--Cahn type equations have seen frequent use in the description of phase separations in binary mixtures, see e.\,g.\ \cite[Cahn and Novick-Cohen]{Ref:Cahn}, and we expect the above model to also exhibit a phase separation into a phase rich in active Ezrin in the rear and a phase poor ofactive Ezrin in the front.

\begin{figure}
	\includegraphics[width=\linewidth]{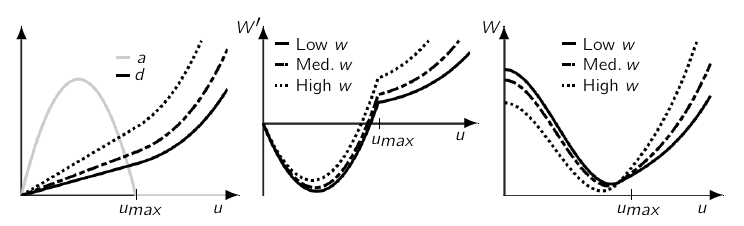}
	\caption{Sketches of the adsorption and desorption rate $a$ and $d$ (left), the corresponding turnover rate $W'$ and the potential $W$ as a function of membranous Ezrin concentration $u$ for different intensities of the velocity $w$ and $\alpha=1$. The potentials have the form of single wells.}
	\label{fig:doubleWell1} 
\end{figure} 

\begin{figure}
	\includegraphics[width=\linewidth]{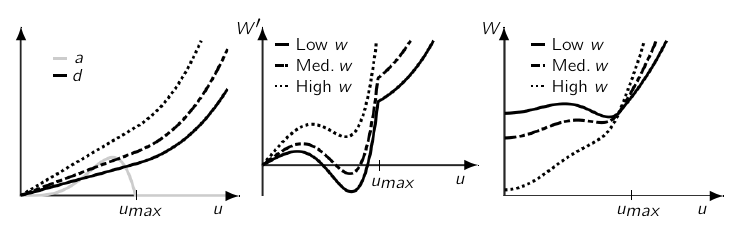}
	\caption{Sketches of the adsorption and desorption rate $a$ and $d$ (left), the corresponding turnover rate $W'$ and the potential $W$ as a function of membranous Ezrin concentration $u$ for different intensities of the velocity $w$ and $\alpha>1$. The potentials have the form of double wells.}
	\label{fig:doubleWell2} 
\end{figure}

\subsection{Nondimensionalisation and model reduction} \label{sec:Nondimensionalisation}

To reduce the number of parameters and to identify the predominant model mechanisms we transform the model equations into a dimensionless form using the following parameter values.
\begin{center}
\begin{tabular}{clcl}
$L$ & cell diameter & 10-20\,$\mu$m & \cite[Braat et al. 1999]{Ref:BraatElAl1999} \\
$L_N$ & diameter of nucleus & 6-10\,$\mu$m & \cite[Braat et al. 1999]{Ref:BraatElAl1999} \\
$T$ & duration of polarization process & 150\,s & own experimental data \\
$\rho$ & density of cytoplasm & 1.03-1.1\,g/cm$^3$ & \cite[Barsanti, Gualtieri 2005]{Ref:BarsantiGualtieri2005} \\
$\rho\lambda$ & dynamic viscosity of cytoplasm & 10$^{-2}$-10$^{-1}$\,Pa\,s & \cite[Mogilner, Manhart 2018]{Ref:MogilnerManhart2018} \\
$c_w$ & typical intracellular velocity & 0.1\,$\mu$m/s & \cite[Mogilner, Manhart 2018]{Ref:MogilnerManhart2018} \\
$\tilde{\kappa}=\frac{\kappa}{\rho\lambda}$ & hydraulic permeability & 0.1\,$\mu$m$^4$/(pN\,s) & \cite[Mogilner, Manhart 2018]{Ref:MogilnerManhart2018} \\
$\nu$ & diffusion rate of active Ezrin & 0.003\,$\mu$m$^2$/s & \cite[Fritzsche et al. 2014]{Ref:FritzscheEtAl2014} \\
$\mu$ & diffusion rate of inactive Ezrin & 30\,$\mu$m$^2$/s & \cite[Coscoy et al. 2002]{Ref:Coscoy} \\
$u_{\max}$ & saturation density of active Ezrin & unknown &  \\
$c_v$ & typical density of inactive Ezrin & unknown &  \\
\end{tabular} 
\end{center}

Subsequently we will indicate dimensionless variables by a hat.
Following the standard procedure, we choose
\begin{equation*}
x = L\hat{x}, \ w = c_w\hat{w}, \ t = \frac{L}{c_w}\hat{t}, \ p = \frac{c_w\rho\lambda L}{\kappa}\hat{p}, \ f = \frac{\rho\lambda c_w}{\kappa}\hat{f}
\end{equation*}
in the incompressible Brinkman--Navier--Stokes equations \eqref{Eq:BNSEquation}-\eqref{Eq:incompressibility} and arrive at
\begin{align*}
\frac{\kappa}{L^2}\Re \Big( \partial_{\hat{t}}\hat{w} + (\hat{w}\cdot\nabla)\hat{w} \Big) + \hat{w} 
&= -\nabla\hat{p} + \frac\kappa{L^2} \Delta\hat{w} + \hat f,\\
\dive\hat w&=0.
\end{align*}
Using using the above-listed parameter values, the Reynolds number is $\Re=\frac{c_wL}{\lambda}\sim1\cdot10^{-8}$-$2\cdot10^{-7}$ and $\kappa/L^2\sim0.0025$-$0.1$.
Thus, keeping only the terms with nonnegligible coefficients, the equations of fluid motion simplify to the Darcy flow equations
\begin{align*}
\hat{w} &= -\nabla\hat{p} + \hat{f}, \\
\dive \, \hat{w} &= 0.
\end{align*}
Note that the observed blebbing time scale $T$ roughly coincides with the time that the observed flow needs to traverse the length of the cell so that the dimensionless final time scales like $\hat T=\frac{Tc_w}L\sim$0.75-1.5.
Further note that even though the bulk force $f$ generated in the actin brushes is unknown, an upper bound can be obtained by multiplying reasonable myosin concentrations (e.\,g.\ $\sim0.2\,\mu$mol/l in plant endoplasm \cite{Ref:Yamamoto2006})
with the force generated per myosin head ($\sim40$\,pN, see e.\,g.\ \cite{Ref:Lohner2018}),
resulting in values of $f\sim5\cdot10^9$ or $\hat f\sim5\cdot10^5$-$5\cdot10^7$ if all myosin molecules were simultaneously active.
Even if only a fraction of the myosin motors is active at any time, the generated force will thus still be able to drive an intracellular flow.

Choosing the same temporal and spatial scales and additionally
\begin{equation*}
u = u_{\max}\hat{u}, \ v = c_v\hat{v},
\end{equation*}
the equation for the active Ezrin concentration turns into
\begin{equation*}
\partial_{\hat{t}}\hat{u} = -\dive_\Gamma(\hat{u}\hat{w}) + \varepsilon\Delta_\Gamma\hat{u} - \frac{L}{c_wu_{\max}}d(c_w\hat{w},u_{\max}\hat{u}) + 
\frac{L}{c_wu_{\max}}a(u_{\max}\hat{u},c_v\hat{v})
\end{equation*}
with $\varepsilon=\frac{\nu}{Lc_w}\sim$0.0015-0.003.
Setting
\begin{align}
\hat{d}(\hat{w},\hat{u}) &= (C_1|\hat w|+C_2)\cdot\begin{cases} \hat{u}  &\text{ if } \hat{u}\leq 1, \\
\frac{\hat u^\zeta+\zeta-1}{\zeta} & \text{ otherwise,} \end{cases} \label{eqn:desorption}\\
\hat{a}(\hat{u},\hat{v}) &= \begin{cases} C_3\hat{u}^\alpha(1-\hat{u})\hat{v} & \text{ if } \hat{u}\leq 1, \\ 0 & \text{ otherwise,} \end{cases}\label{eqn:adsorption}
\end{align}
for the dimensionless parameters
\begin{equation*}
C_1 = L\beta_1, \ C_2 = \frac{L}{c_w}\beta_2, \ C_3 = \frac{Lu_{\max}^\alpha c_v}{c_w}\gamma,
\end{equation*}
the equation finally reduces to 
\begin{align*}
\partial_{\hat{t}}\hat{u} &= -\dive_\Gamma(\hat{u}\hat{w}) + \varepsilon\Delta_\Gamma\hat{u} - \hat{d}(\hat{w},\hat{u}) + \hat{a}(\hat{u},\hat{v}).
\end{align*}
We will keep the $\varepsilon$-weighted diffusion term as a regularization; from the theory of Allen--Cahn type equations it is known that it governs the width of the diffusive interface between phases of different Ezrin concentration.
Due to a lack of experimental information about typical concentrations $u_{\max}$ and $c_v$ of active and inactive Ezrin as well as the exact form of the adsorption and desorption rate, the dimensionless model equations still involve some unknown parameters. We will present some possible choices and discuss the influence of different parameters within \cref{sec:ExperimentalResults}.


To reduce the equations for the inactive Ezrin concentration we will make the assumption that
\begin{equation*}
c_vL\gg u_{\max}
\qquad\text{for the typical inactive Ezrin concentration }
c_v=\frac1{|\Omega|}\int_\Omega v_0\,\d x,
\end{equation*}
where $|\Omega|$ denotes the Lebesgue measure of $\Omega$.
Unfortunately, in the literature we were not able to find any measurements of Ezrin concentrations in cells in order to support this assumption.
Nevertheless we believe the assumption to be reasonable and to represent the typical situation for most membrane-active proteins in a cell,
since assuming $c_vL\sim u_{\max}$ would mean that the typical number of protein molecules in the cytosol equals the typical number in the membrane,
meaning a (highly unlikely) perfect recruitment to the membrane.
We now introduce the spatial concentration average and residual
\begin{equation*}
\bar{v} = \frac{1}{|\Omega|}\int_\Omega v \, \d x, \quad r = v-\bar{v}.
\end{equation*}
Using \eqref{Eq:pdeV}-\eqref{Eq:BCV}, the functions $\bar v$ and $r$ are governed by the differential equations
\begin{align*}
\frac{\d\bar{v}}{\d t} &= \frac{1}{|\Omega|} \int_\Omega \mu\Delta v- \dive(vw) \, \d x 
= \frac{1}{|\Omega|} \int_\Gamma D-A \, d\mathcal{H}^2, &&\text{on }(0,T)\times\Omega,\\
\partial_t r &= \mu\Delta r - \dive(rw) - \frac{1}{|\Omega|}\int_\Gamma D-A \, d\mathcal{H}^2 &&\text{on }(0,T)\times\Omega, \\
\mu\nabla r\cdot n &= D-A &&\text{on }(0,T)\times\Gamma, \\
r(0,\cdot) &= r_0 &&\text{on }\Omega, \\
\bar v(0)&=\bar v_0,\\
\int_\Omega r \, \d x &= 0
\end{align*}
for $\bar v_0=\int_\Omega v_0\,\d x/|\Omega|$ and $r_0=v_0-\bar v_0$,
where we used integration by parts and $w\cdot n=0$ on $\Gamma$, and where $\hd^2$ denotes the surface (two-dimensional Hausdorff) measure.
Introducing
\begin{equation*}
r = c_v\hat{r},\
\bar v=c_v\hat v
\end{equation*}
as well as the P\'eclet number $\Pe=\frac{c_wL}{\mu}\sim$0.03-0.07$\ll1$, the differential equation for $\hat r$ becomes
\begin{align*}
\Pe\,\partial_{\hat{t}}\hat{r} &= \Delta\hat{r} - \Pe\,\dive(\hat{w}\hat{r}) - \frac{1}{|\hat{\Omega}|}\frac{L^2}{\mu}\frac{u_{\max}}{Lc_v}\int_{\hat{\Gamma}} \frac{D-A}{u_{\max}} \, d\mathcal{H}^2 &&\text{on }(0,T)\times\Omega, \\
\nabla\hat r\cdot n &= \frac{L^2}{\mu}\frac{u_{\max}}{Lc_v}\frac{D-A}{u_{\max}} &&\text{on }(0,T)\times\Gamma.
\end{align*}
Our biological experiments suggest that it takes roughly around 10\,s until a substantial Ezrin concentration establishes on an originally Ezrin-depleted membrane so that the Ezrin turnover rate can be estimated to satisfy
\begin{equation*}
\frac{|D-A|}{u_{\max}}\leq\frac1{10\,\text s}\,.
\end{equation*}
Together with $c_vL\gg u_{\max}$ we thus obtain $\frac{L^2}{\mu}\frac{u_{\max}}{Lc_v}\frac{|D-A|}{u_{\max}}\sim\frac{u_{\max}}{Lc_v}\ll1$ so that, by neglecting terms with small coefficients, the equations for $\hat r$ turns into 
\begin{equation*}
\Delta\hat r = 0 \text{ on }(0,T)\times\Omega, \ \nabla\hat r\cdot n =0 \text{ on }(0,T)\times\Gamma, \ \int_\Omega\hat r \, \d x = 0,
\end{equation*}
and thus $\hat r=0$. Therefore, we may assume that $v$ remains constant in space, $v=\bar{v}$.
Now it is obvious that the total Ezrin molecule number within the cell stays constant over time.
Indeed, we have
\begin{equation*}
\frac\d{\d t}\left(|\Omega|\bar v+\int_\Gamma u\,\d\hd^2\right)=\int_\Gamma D-A\,\d\hd^2+\int_\Gamma\partial_tu\,\d\hd^2=\int_\Gamma\nu\Delta_\Gamma u-\dive_\Gamma(uw)\,\d\hd^2=0
\end{equation*}
using integration by parts.
Furthermore, if $u_0$ is nonnegative, $u$ will be so for all times so that $|\Omega|\bar v\leq\int_\Omega v_0\,\d x+\int_\Gamma u_0\,\d\hd^2=|\Omega|c_v+\int_\Gamma u_0\,\d\hd^2$.
Likewise, $|\Omega|\bar v\geq|\Omega|c_v-\max_{t\in(0,T)}\int_\Gamma u\,\d\hd^2$, where we can estimate $\int_\Gamma u\,\d\hd^2\lesssim\hd^2(\Gamma)u_{\max}$ due to the strong desorption of Ezrin for $u>u_{\max}$.
In summary,
\begin{equation*}
\left|\frac{\bar v}{c_v}-1\right|=\frac1{|\Omega|c_v}||\Omega|\bar v-|\Omega|c_v|\lesssim\frac{\hd^2(\Gamma)u_{\max}}{|\Omega|c_v}\sim\frac{u_{\max}}{c_vL}\ll1
\end{equation*}
so that we may well approximate
\begin{equation*}
\bar v=c_v\,.
\end{equation*}


In summary, the reduced dimensionless system of equations read
\begin{align}
\hat w &= -\nabla\hat p +\hat f&&\text{ on } (0,\hat T)\times\hat\Omega,\label{eqn:Darcy1}\\
\dive\hat w &=0&&\text{ on } (0,\hat T)\times\hat\Omega,\label{eqn:Darcy2}\\
\hat w\cdot n &= 0 &&\text{ on } (0,\hat T)\times\hat\Gamma,\label{eqn:DarcyBC}\\
\partial_{t}\hat u &= -\dive_{\hat\Gamma}(\hat u\hat w) + \varepsilon\Delta_{\hat\Gamma}\hat u - \hat{d}(\hat w,\hat u) + \hat{a}(\hat u,1)&&\text{ on } (0,\hat T)\times\hat\Gamma,\label{eqn:Ezrin}\\
\hat u(0,\cdot) &=\hat u_0 &&\text{ on } \hat\Gamma,\label{eqn:EzrinBC}
\end{align}
in which all quantities can be expected to be roughly of size 1 and in which $\hat d$ and $\hat a$ are defined by \eqref{eqn:desorption}-\eqref{eqn:adsorption}.


\section{Existence and regularity of the solution} \label{sec:Analysis}

For ease of notation, from now on we remove the hat on the dimensionless variables.
The reduced model consists of two differential equation systems, one for the velocity field $w$, which is independent of $u$ and thus can be solved separately,
and one for the membranous Ezrin concentration $u$, into which the velocity $w$ enters via advection and as a parameter in the desorption term.
Both sets of equations are standard, and in this \namecref{sec:Analysis} we briefly summarize the well-posedness of the model.

For a bounded, relatively open domain $A$ we will denote by $L^q(A)$ and $W^{m,q}(A)$ the real Lebesgue and Sobolev spaces on $A$ with exponent $q$ and (potentially fractional or negative) differentiability order $m$;
in case of $X$-valued function spaces for some vector space $X$ we indicate the image space via $L^q(A;X)$ and $W^{m,q}(A;X)$, respectively (the same notation is used for the following function spaces).
The corresponding Hilbert spaces are called $H^m(A)=W^{m,2}(A)$.
The space of $m$ times continuously differentiable functions on $\overline A$ is denoted $C^m(\overline A)$ and of the corresponding H\"older-differentiable functions with exponent $\alpha$ by $C^{m,\alpha}(\overline A)$.
Finally, the space of Radon measures on $\overline A$ (the dual space to $C^0(\overline A)$) is denoted $\meas(\overline A)$.

We will either use $A=\Omega$ or $A=\Gamma$.
In the latter case, recall that the function spaces can be defined on geodesically complete compact manifolds via charts, that is,
we consider a finite atlas $(U_i, x_i)_{1\leq i \leq N}$ of $\Gamma$ and a smooth partition $\psi_i:U_i\to[0,1]$, $i=1,\ldots,N$, of unity on $\Gamma$
so that the function space $X(\Gamma)$ is defined as the set of all functions $g:\Gamma\to\R$ such that $(\psi_ig)\circ x_i^{-1}\in X(x_i(U_i))$,
and the corresponding norm is defined as
\begin{equation*}
\|g\|_{X(\Gamma)}=\sum_{i=1}^N \|(\psi_i v)\circ x^{-1}_i\|_{X(x_i(U_i))}\,.
\end{equation*}
As long as the charts are bi-Lipschitz (which is the case for Lipschitz $\Gamma$), all spaces $C^{0,\alpha}(\Gamma)$ and $W^{m,p}(\Gamma)$ with $m\leq1$ are well-defined by this procedure,
that is, the norm depends on the chosen atlas, but all atlases lead to equivalent norms.
The definition of more regular function spaces on $\Gamma$ then requires correspondingly smoother $\Gamma$.

\subsection{Influence of the bulk force on Darcy flow}
For the well-posedness of the differential equation for membranous Ezrin we will have to bound the influence of the advection term.
For this we will require a sufficiently smooth flow $w$ tangentially to the membrane.
Using standard results for elliptic differential equations, this could be achieved by requiring sufficiently high regularity of the bulk force $f$ and the domain boundary $\Gamma$.
However, recalling that the bulk force is produced by the actin brushes in the cell interior, we believe it more natural to obtain the necessary flow regularity from the positive distance of $f$ to the cell boundary.
As for the regularity of $\Gamma$, we only assume that the cell $\Omega$ is convex, which before blebbing is certainly a reasonable assumption.
The convexity entails that $\Gamma$ is Lipschitz, the minimum requirement to make sense of (weak) gradients of functions on $\Gamma$ and thus of the differential equation \eqref{eqn:reactionAdvectionDiffusion} for Ezrin.

For simplicity, we will sometimes consider the second-order form of Darcy's equations,
\begin{align}
\Delta p&=\dive f
&&\text{in }\Omega,\label{eqn:DarcyP}\\
\nabla p\cdot n&=f\cdot n
&&\text{on }\Gamma,\label{eqn:DarcyPBC}
\end{align}
which is equivalent to \eqref{eqn:Darcy1}-\eqref{eqn:DarcyBC} by taking the divergence in \eqref{eqn:Darcy1}.

\begin{lemma}[Fundamental solution]
Let $f=D\delta_z$ for $D\in\R^n$ and $\delta_z$ the Dirac distribution in $z\in\R^n$, $n>1$.
Then \eqref{eqn:DarcyP} is solved in the distributional sense by $p=D\cdot \Psi(\cdot-z)$ with
\begin{equation*}
\Psi(x)=\frac1{\omega_n}\frac{x}{|x|^n}\,,
\end{equation*}
where $\omega_n$ denotes the surface area of the $n$-dimensional unit ball.
\end{lemma}
\begin{proof}
This can be checked by straightforward calculation or by noticing that $\Psi$ is nothing else than the gradient of the funcamental solution to Poisson's equation, however,
it is also well-known in electro-encephalography research, where $D\cdot\Psi$ is the electric field induced by a voltage dipole as they occur in the human brain \cite[(12)]{Ref:Sarvas1987}.
\end{proof}

\begin{theorem}[Existence and regularity]\label{thm:existenceDarcy}
Let $\Omega\subset\R^n$ be convex, let $f\in\meas(\overline\Omega;\R^n)$ be compactly supported in $\Omega$, $n>1$.
Then there exists a (unique up to a constant) distributional solution $p$ to \eqref{eqn:DarcyP}-\eqref{eqn:DarcyPBC},
which satisfies
\begin{equation*}
p(x)=\tilde p(x)+\int_\Omega\Psi(x-z)\cdot\d f(z)
\end{equation*}
for some $\tilde p\in C^{0,1}(\overline\Omega)$, and for any $s>n$ there exists a constant $C(s,n,\Omega)$ with
\begin{equation*}
\|\tilde p\|_{C^{0,1}(\overline\Omega)}
\leq C(s,n,\Omega)\dist(\spt f,\Gamma)^{-s}\|f\|_{\meas(\overline\Omega)}\,.
\end{equation*}
As a consequence, the unique distributional solution $(w,p)$ to \eqref{eqn:Darcy1}-\eqref{eqn:DarcyBC} satisfies
\begin{equation*}
\|w\|_{L^\infty(\Gamma)}\leq C(s,n,\Omega)\dist(\spt f,\Gamma)^{-s}\|f\|_{\meas(\overline\Omega)}\,.
\end{equation*}
\end{theorem}
\begin{proof}

Abbreviate $R=\dist(\spt f,\Gamma)$, let $\chi:\R^n\to[0,1]$ be a smooth radially symmetric function with $\chi(x)=1$ for $|x|<1/2$ and $\chi(x)=0$ for $|x|>1$, and set $\chi_R(x)=\chi(x/R)$ as well as
\begin{equation*}
P(x)=\int_\Omega[(1-\chi_R)\Psi](x-z)\cdot\d f(z)\,.
\end{equation*}
As a convolution of a measure with an infinitely smooth function, $P$ is smooth.
Furthermore, let $\hat p\in W^{1,2}(\Omega)$ denote the unique weak solution with zero mean to
\begin{equation*}
\Delta\hat p=\Delta P\quad\text{in }\Omega,\qquad
\nabla\hat p\cdot n=0\quad\text{on }\Gamma.
\end{equation*}
Then
\begin{equation*}
p(x)=\hat p(x)+\int_\Omega[\chi_R\Psi](x-z)\cdot\d f(z)
\end{equation*}
solves (uniquely up to a constant) \eqref{eqn:DarcyP}-\eqref{eqn:DarcyPBC}.
Now by the Lipschitz regularity result \cite[\S2]{Ref:Mazya2009} for the Neumann problem on convex domains we have
\begin{equation*}
\|\hat p\|_{C^{0,1}(\overline\Omega)}\leq\tilde C\left\|\Delta P\right\|_{L^q(\Omega)}
\end{equation*}
for any $q>n$, where the constant $\tilde C>0$ depends on $\Omega$, $n$, and $q$. Now by Young's convolution inequality we can estimate
\begin{align*}
\left\|\Delta P\right\|_{L^q(\Omega)}
&=\left[\int_\Omega\left|\int_\Omega\Delta\left[(1-\chi_R)\Psi\right](x-z)\cdot\d f(z)\right|^q\d x\right]^{\frac1q}\\
&\leq\left[\int_\Omega\left(\int_\Omega\left|\Delta\left[(1-\chi_R)\Psi\right](x-z)\right|\d\frac{|f|}{\|f\|_{\meas(\overline\Omega)}}(z)\right)^q\d x\right]^{\frac1q}\|f\|_{\meas(\overline\Omega)}\\
&\leq\left[\int_\Omega\int_\Omega\left|\Delta\left[(1-\chi_R)\Psi\right](x-z)\right|^q\d\frac{|f|}{\|f\|_{\meas(\overline\Omega)}}(z)\,\d x\right]^{\frac1q}\|f\|_{\meas(\overline\Omega)}\\
&\leq\left[\int_\Omega\int_\Omega\left|\Delta\left[(1-\chi_R)\Psi\right](x-z)\right|^q\d x\,\d\frac{|f|}{\|f\|_{\meas(\overline\Omega)}}(z)\right]^{\frac1q}\|f\|_{\meas(\overline\Omega)}\\
&=\|\Psi\Delta\chi_R+2\nabla\chi_R\cdot\nabla\Psi\|_{L^q(\R^n)}\|f\|_{\meas(\overline\Omega)}\\
&\leq\|\chi\|_{C^2(\R^n)}\left(R^{-2}\|\Psi\|_{L^q(\R^n\setminus B_R(0))}+R^{-1}\|\nabla\Psi\|_{L^q(\R^n\setminus B_R(0))}\right)\|f\|_{\meas(\overline\Omega)}\\
&\leq\hat CR^{-n-1+n/q}\|f\|_{\meas(\overline\Omega)}
\end{align*}
for some constant $\hat C$ depending only on $\chi$, $\Omega$, $n$, and $q$.
Summarizing, we have
\begin{equation*}
\|\hat p\|_{C^{0,1}(\overline\Omega)}\leq\tilde C\hat CR^{-n-1+n/q}\|f\|_{\meas(\overline\Omega)}\,.
\end{equation*}
Now $\tilde p(x)=p(x)-\int_\Omega\Psi(x-z)\cdot\d f(z)=\hat p(x)-P(x)$ so that
\begin{equation*}
\|\tilde p\|_{C^{0,1}(\overline\Omega)}
\leq\|\hat p\|_{C^{0,1}(\overline\Omega)}+\|P\|_{C^{0,1}(\overline\Omega)}
\leq\tilde C\hat CR^{-n-1+n/q}\|f\|_{\meas(\overline\Omega)}+\bar CR^{-n}\|f\|_{\meas(\overline\Omega)}
\end{equation*}
for some $\bar C>0$ depending on $\chi$. Due to the arbitrariness of $q>n$ this proves the first claim.

As for the estimate on $w$, we also have
\begin{equation*}
\|p\|_{W^{1,\infty}(\Gamma)}=\|p\|_{C^{0,1}(\Gamma)}=\|\hat p\|_{C^{0,1}(\Gamma)}\leq\|\hat p\|_{C^{0,1}(\Omega)}\leq CR^{-n-1+n/q}\|f\|_{\meas(\overline\Omega)}\,,
\end{equation*}
from which the second claim follows by noting $\|w\|_{L^\infty(\Gamma)}=\|\nabla p\|_{L^\infty(\Gamma)}\leq\|p\|_{W^{1,\infty}(\Gamma)}$ (recall that $\nabla p$ is tangential to $\Gamma$).
\end{proof}

\begin{remark}[Regularity of flux at boundary]\label{rem:DarcyRegularity}
Due to $w=-\nabla p+f$, the above result shows that $w$ is in $L^\infty(\Omega\setminus B_\delta(\spt f))$ for $B_\delta$ the $\delta$-neighbourhood,
and the boundary flux is its trace on $\Gamma$ (in the sense that it is the weak-$\ast$ limit of $w$ on surfaces $\Gamma_n\subset\Omega$ approaching $\Gamma$).
By another regularity result for the Neumann problem on convex domains we additionally have $p\in W^{2,2}(\Omega\setminus B_\delta(\spt f))$ so that for $n=3$ we have $w\in H^{1/2}(\Gamma)$ \cite{AdolfssonJerison}.
If higher regularity of $\Gamma$ is assumed, then the boundary flux $w|_\Gamma$ may be even smoother.
In our exposition, though, we will stick to minimal regularity requirements on $\Gamma$.
\end{remark}

On a ball one can explicitly compute a Green's function and thus state an analytical solution, which can provide some intuition on the flow behaviour.

\begin{proposition}[Green's function]\label{thm:DarcyGreenBall}
Equations \eqref{eqn:DarcyP}-\eqref{eqn:DarcyPBC} for $f=D\delta_z$ on $\Omega=B_1(0)\subset\R^3$ the unit ball are solved by $p=D\cdot G_z$ for the Green's function
\begin{multline*}
G_z(x)
=\frac1{4\pi}\left[\left(\frac{x\cdot e-|z|}{|x-z|^3}-\frac{x\cdot e-\frac1{|z|}}{|z|^3|x-\zeta|^3}\right)e\right.\\
\left.+\frac1{|z||x-\zeta|}\left(1+\frac1{|z|^2|x-\zeta|^2}+\frac{|z||x-\zeta|}{|x-z|^3}+\frac{x\cdot e}{\frac1{|z|}-x\cdot e+|x-\zeta|}\right)\left(I-e\otimes e\right)x\right]\,,
\end{multline*}
where $\zeta=z/|z|^2$ and $e=z/|z|$.
Thus, the solution for general $f$ is given by
\begin{equation*}
p(x)=\int_\Omega G_z(x)\cdot\d f(z)\,.
\end{equation*}
\end{proposition}
\begin{proof}
One can check by explicit calculation that $G_z$ solves the desired system.
Alternatively, one can exploit $G_z(x)=-\nabla_z\tilde G_z(x)$, where $\tilde G_z$ is the Green's function for the Neumann problem,
$\Delta\tilde G_z=\delta_z-\frac3{4\pi}$ in $\Omega$ with $\nabla\tilde G_z\cdot n=0$ on $\partial\Omega$,
which is for instance stated in \cite[\S7.1.2c]{Ref:DiBenedetto2010} (note that the first term in that reference should be corrected to have a negative sign).
\end{proof}

\begin{remark}[Regions of maximal boundary velocity in a round cell]
\Cref{thm:existenceDarcy} provides an upper bound on the maximum boundary velocity, which diverges as the forces $f$ approach the boundary $\Gamma$.
Of course, this does not automatically imply the blow-up of the boundary velocity as the forces get closer to $\Gamma$
(indeed, the boundary velocity will for instance be zero if the forces are arranged within a small ball such that the net force outside is zero).
However, within a cell we do not expect force cancellation
so that a rough intuition of the flow can be obtained by approximating the actin brushes with a point force $f$.
Then, approximating the cell by a ball $\Omega=B_1(0)$ with the cell front being the northpole $e_3=(0\,0\,1)^T$ and the actin brush $f=e_3\delta_{ce_3}$ underneath for some $c\in(0,1)$,
the analytical solution from \cref{thm:DarcyGreenBall} simplifies to
\begin{equation*}
p(x)
=\frac1{4\pi}\left(\frac{x_3-c}{|x-e_3c|^3}-\frac{x_3-1/c}{c^3|x-e_3/c|^3}\right)\,.
\end{equation*}
The resulting velocity $w$ at the boundary $\partial B_1(0)$ is given by
\begin{equation*}
w(x)=\frac3{4\pi}\frac{1-c^2}{|x-z|^5}(x_3x-e_3)\,,
\end{equation*}
pointing south along great circles.
Its magnitude is given by
\begin{equation*}
|w(x)|=\frac3{4\pi}\frac{1-c^2}{|x-z|^5}\sqrt{1-x_3^2}
=\frac3{4\pi}\frac{1-c^2}{\sqrt{(x_3-c)^2+1-x_3^2}^5}\sqrt{1-x_3^2}\,,
\end{equation*}
where $1-x_3^2$ is the squared distance of $x$ to the vertical axis.
This is maximized at $x_3=(\sqrt{(c^2+1)^2+60c^2}-(c^2+1))/6c\sim1-(c-1)^2/8=1-\dist(\spt f,\Gamma)^2/8$, for which $\sqrt{1-x_3^2}\sim\dist(\spt f,\Gamma)/\sqrt8$.
In other words, the point of maximal boundary velocity occurs roughly $\dist(\spt f,\Gamma)/\sqrt8$ away from the cell front,
and the velocity there indeed scales like $\dist(\spt f,\Gamma)^{-3}$.
\end{remark}

\subsection{Well-posedness of the Ezrin equation}

The well-posedness of \eqref{eqn:Ezrin}-\eqref{eqn:EzrinBC} can be obtained via classical semigroup theory (we will later briefly discuss other approaches as well).
To this end we abbreviate
\begin{align*}
L&=\varepsilon\Delta_\Gamma
\qquad\text{and}\\
F_w(u) &= -\mathrm{div}_{\Gamma}(uw) - d(w,u) + a(u,1)
\end{align*}
so that \eqref{eqn:Ezrin}-\eqref{eqn:EzrinBC} turn into
\begin{align}
\partial_t u&=Lu+F_w(u)&&\text{on }(0,T)\times\Gamma\,,\label{eqn:reactionAdvectionDiffusion}\\
u(0,\cdot)&=u_0&&\text{on }\Gamma\,.\label{eqn:reactionAdvectionDiffusionBC}
\end{align}
Furthermore, the bounded operator which maps an element $v_0\in L^q(\Gamma)$ to the unique solution $v$ of the Cauchy problem 
\begin{equation*}
  \partial_t v = Lv,\quad
  v(0,\cdot) = v_0,
\end{equation*}
will be denoted by $e^{tL}$, thus $v(t,\cdot)=e^{tL}v_0$.
The short-term existence of a solution to \eqref{eqn:reactionAdvectionDiffusion}-\eqref{eqn:reactionAdvectionDiffusionBC} is based on the following Lipschitz property of $F_w$.

\begin{lemma}[Lipschitz property of $F_w$]\label{thm:LipschitzProperty}
Let $\Gamma=\partial\Omega\subset\R^3$ be bounded and Lipschitz and $w\in L^\infty(\Gamma)$.
$F_w$ is locally Lipschitz from $L^q(\Gamma)$ into $W^{-1,q}(\Gamma)$ for any $q\geq\max\{\zeta,2\zeta-2\}$,
where $\zeta>1$ is the exponent in \eqref{eqn:desorption}.
\end{lemma}
\begin{proof}
By H\"older's inequality we have
\begin{equation*}
\|uw\|_{L^q(\Gamma)}\leq\|u\|_{L^q(\Gamma)}\|w\|_{L^\infty(\Gamma)}
\end{equation*}
so that $\|\dive(uw)\|_{W^{-1,q}(\Gamma)}\leq n\|w\|_{L^\infty(\Gamma)}\|u\|_{L^q(\Gamma)}$.
Furthermore, letting $\ell$ denote the Lipschitz constant of $a(\cdot,1)$, we obtain
\begin{equation*}
\|a(u_1,1)-a(u_2,1)\|_{W^{-1,q}(\Gamma)}
\leq\|a(u_1,1)-a(u_2,1)\|_{L^q(\Gamma)}
\leq\ell\|u_1-u_2\|_{L^q(\Gamma)}\,.
\end{equation*}
Finally, abbreviating $r=\frac q\zeta\geq1$ we have by H\"older's inequality
\begin{equation*}
\|d(w,u)\|_{L^r(\Gamma)}
\leq\|C_1w+C_2\|_{L^\infty(\Gamma)}\|1+|u|^\zeta\|_{L^r(\Gamma)}
\leq\|C_1w+C_2\|_{L^\infty(\Gamma)}\left(\|1\|_{L^r(\Gamma)}+\|u\|^\zeta_{L^q(\Gamma)}\right)\,,
\end{equation*}
and the directional derivative of $d(w,u)$ in direction $\phi\in L^q(\Gamma)$ satisfies
\begin{align*}
\|\partial_ud(w,u)(\phi)\|_{L^r(\Gamma)}
&=\left\|(C_1|w|+C_2)\max\{1,\mathrm{sign}(u)|u|^{\zeta-1}\}\phi\right\|_{L^r(\Gamma)}\\
&\leq\|C_1w+C_2\|_{L^\infty(\Gamma)}\|\phi\|_{L^q(\Gamma)}\|1+|u|^{\zeta-1}\|_{L^{\frac{q}{\zeta-1}}(\Gamma)}\\
&\leq\|C_1w+C_2\|_{L^\infty(\Gamma)}\|\phi\|_{L^q(\Gamma)}\left(\|1\|_{L^{\frac{q}{\zeta-1}}(\Gamma)}+\|u\|_{L^q(\Gamma)}^{\zeta-1}\right)\,.
\end{align*}
Consequently, $d(w,\cdot)$ lies in $L^r(\Gamma)$ and is locally Lipschitz (even G\^ateaux differentiable) from $L^q(\Gamma)$ into $L^r(\Gamma)$.
Due to the Sobolev embedding of $L^r(\Gamma)$ into $W^{-1,q}(\Gamma)$ we arrive at the desired result.
\end{proof}

\begin{theorem}[Short-term existence]\label{Thm:ShortTermExistence}
Let $\Gamma$, $w$, and $q>2$ as in \cref{thm:LipschitzProperty} and $u_0\in L^q(\Gamma)$.
Then there exists $\hat t > 0$ depending on $\Gamma$ and $\|u_0\|_{L^q(\Gamma)}$ such that the initial value problem
\begin{equation}\label{Eq:ProblemAbstract}
  \partial_t u = Lu + F_w(u), \quad u(0,\cdot) = u_0
\end{equation}
on $\Gamma$ has a unique weak solution $u\in C^0([0,\hat t]; L^q(\Gamma))\cap C^0((0,\hat t]; C^{0,\alpha}(\Gamma))$ for any $\alpha\in[0,1-2/q)$.
\end{theorem}
\begin{proof}
 Following the exposition in \cite[Taylor, Ch.\,15.1]{Ref:Taylor}, the result is obtained by studying the integral equation
\begin{equation*}\label{Eq:IntegralEquationSolution}
  \Psi(u(t,\cdot)) = u(t,\cdot)
  \quad\text{ with }
  \Psi(u(t,\cdot)) = e^{tL}u_0 + \int_0^t e^{(t-s)L}F_w(u(s,\cdot))\, \d s
\end{equation*}
and treating it as a fixed point equation.
Analogously to the proof of \cite[Taylor, Ch.\,15.1, Prop.\,1.1]{Ref:Taylor}, the existence of a fixed point can be ascertained via the contraction mapping principle, given that for appropriate Banach spaces $X$ and $Y$
\begin{equation*}
\begin{aligned}
 e^{tL}&:X \to X &&\text{is a }C^0\text{-semigroup for }t \geq 0, \\
 F_w&:X \to Y &&\text{is locally Lipschitz,} \\
 e^{tL}&:Y \to X&&\text{with }\|e^{tL}\|_{\mathcal{L}(Y,X)} \leq Ct^{-\xi} \text{ for some } C>0,\,\xi <1\text{ and all } t \in (0, 1] .\hspace*{-3ex}
\end{aligned}
\end{equation*}
Above, $\|\cdot\|_{\mathcal L(Y,X)}$ denotes the operator norm of a linear operator from $Y$ to $X$.

For our choice $L = \varepsilon \Delta$, standard results for linear parabolic equations (see e.\,g.\ \cite[Ch.\,6.1]{Ref:TaylorFirstVolume}) guarantee
that the semigroup generated by $L$ is well-defined and strongly continuous. 
The appropriate Banach spaces $X$ and $Y$ are determined by the properties of $F_w$ and the obtainable bounds on the operator norm of $e^{tL}$.
For $X = W^{s,p}(\Gamma)$ and $Y=W^{r,q}(\Gamma)$ with $q \leq p$ and $r\leq s$, \cite[Ch.\,6.1]{Ref:TaylorFirstVolume} states
\begin{equation}\label{Eq:NormEstimateOperatorNorm}
\|e^{tL}\|_{\mathcal{L}(Y,X)} \leq C t^{ -\frac{1}{2}(\frac{1}{q}-\frac{1}{p})-\frac{1}{2}(s-r)}\,.
\end{equation}
In our case we pick $X = L^q(\Gamma)$, $Y = W^{-1,q}(\Gamma)$ so that $\xi = \frac12$ and the map $F_w:X\to Y$ is well-defined and locally Lipschitz continuous by \cref{thm:LipschitzProperty}.
For small enough $\hat t > 0$, which might still depend on the Lipschitz constant of $F_w$, $\xi$ and $\|u_0\|_{L^q(\Gamma)}$, we then obtain the existence and uniqueness of a solution $u \in C([0,\hat t]; X)$ to \eqref{Eq:ProblemAbstract}.

Higher spatial regularity can be derived from the estimate
\begin{equation}
 \|u(t,\cdot)\|_{V}  \leq \|e^{(t-t_0)L}\|_{\mathcal{L}(W,V)} \|u(t_0,\cdot)\|_{W} + \int_{t_0}^t \|e^{(t-s)L}\|_{\mathcal{L}(W,V)} \|F_w(u(s,\cdot))\|_{W}\, \d s
\end{equation}
with appropriate Banach function spaces $V$ and $W$ on $\Gamma$ and any $t_0\in[0,\hat t]$.
Choosing $W=Y=W^{-1,q}(\Gamma)$ and $V=W^{r,q}(\Gamma)$ with $r<1$ we obtain $\|e^{tL}\|_{\mathcal{L}(Y,X)} \leq Ct^{-\xi}$ with $\xi=\frac{1}{2}(r+1)<1$
so that $u\in C^0((0,\hat t];W^{r,q}(\Gamma))$.
In particular, by the continuous embedding $W^{r,q}(\Gamma)\subset C^{0,r-2/q}(\Gamma)$ \cite[Thm.\,8.2]{Ref:NezzaPalatucciValdinoci2012} we obtain $u\in C^0((0,\hat t];C^{0,r-2/q}(\Gamma))$ for any $r\in(\frac2q,1)$.
%
%
%
\end{proof}

\begin{remark}[Higher spatial regularity]
For smoother $\Gamma$ (and thus also smoother $w$, see \cref{rem:DarcyRegularity}), the spatial regularity of $u$ can be improved by consecutively using different pairs $V$, $W$ in the above argument.
For instance, one could study the action of $F_w$ and $e^{tL}$ on the following sequence of spaces,
\begin{multline*}
L^q(\Gamma) \stackrel{F_w}{\longrightarrow} H^{-1}(\Gamma) \stackrel{ e^{tL} }{\longrightarrow} H^{\frac{1}{2}} (\Gamma) \stackrel{F_w}{\longrightarrow} H^{-\frac{1}{2}} (\Gamma) \stackrel{e^{tL}}{\longrightarrow} H^{1}(\Gamma) \ldots \\ 
\ldots  \stackrel{ F_w }{\longrightarrow} L^2(\Gamma) \stackrel{e^{tL}}{\longrightarrow} H^{\frac{3}{2}}(\Gamma) \stackrel{F_w}{\longrightarrow} H^{\frac{1}{2}}(\Gamma)  \stackrel{e^{tL}}{\longrightarrow} H^{2}(\Gamma)\,,
\end{multline*}
which is for instance valid for twice differentiable $\Gamma$ and $w$ and would yield $u\in C^0((0,\hat t];H^2(\Gamma))$.
\end{remark}

\begin{proposition}[Global boundedness]\label{thm:globalBound}
Let $\Gamma$, $w$, $q$, $\alpha$, $u_0$ and $\hat t$ as in \cref{Thm:ShortTermExistence}, assume $u_0>0$,
and let $u$ be the weak solution to \eqref{eqn:Ezrin}-\eqref{eqn:EzrinBC} on $[0,\hat t]\times\Gamma$.
Then there exists some $M>0$ only depending on $\Gamma$, $w$, and $q$ with
$u(t,\cdot)\geq0$ and $\|u(t,\cdot)\|_{L^q(\Gamma)}\leq\max\{\|u_0\|_{L^q(\Gamma)},M\}$ for all $t\in(0,\hat t)$.
\end{proposition}
\begin{proof}
We would like to prove the above via a priori estimates obtained from testing the differential equation with different functions,
however, for this we need $\partial_t u$ to be sufficiently regular.
We achieve this by mollification in time.

Let $G:\R\to[0,\infty)$ be a smooth mollifier with support on $[-1,1]$ and mass $\int_\R G(t)\,\d t=1$.
For $\delta>0$ define $G_\delta(t)=G(t/\delta)/\delta$, and let $u_\delta(\cdot,x)=G_\delta\ast u(\cdot,x)$ for $x\in\Gamma$ be the mollification of the solution $u$ in time so that $u_\delta\in C^\infty((\delta,\hat t-\delta),C^{0,\alpha}(\Gamma))$.
Note that $u_\delta$ satisfies
\begin{equation*}
\partial_tu_\delta=\dive_\Gamma(\varepsilon\nabla_\Gamma u_\delta-u_\delta w)+G_\delta\ast(a(u,1)-d(w,u))\,.
\end{equation*}

Now assume that $u$ does not stay positive.
Since $u$ is continuous, there is a time
\begin{equation*}
t_\pm=\inf\left\{t\in(0,\hat t)\,\middle|\,\min u(t,\cdot)<0\right\}
\end{equation*}
when $u$ changes sign for the first time.
Take $\delta<\min\{t_\pm,\hat t-t_\pm\}/2$ small enough such that also $u_\delta$ changes sign at some time $t_\delta\in(2\delta,\hat t-2\delta)$.
Then necessarily, $\int_\Gamma u_\delta(t,x)^{-2/\alpha}\,\d\hd^2(x)\to\infty$ as $t\to t_\delta$,
since this integral must be infinite for any nonnegative $C^{0,\alpha}(\Gamma)$ function taking the value $0$ at some $x\in\Gamma$.
However,
\begin{align*}
\frac\alpha2\frac\d{\d t}\int_\Gamma u_\delta^{-\frac2\alpha}\,\d\hd^2
&=-\int_\Gamma u_\delta^{-\frac2\alpha-1}\partial_t u_\delta\,\d\hd^2\\
&=-\int_\Gamma u_\delta^{-\frac2\alpha-1}\left(\dive_\Gamma(\varepsilon\nabla_\Gamma u_\delta-u_\delta w)+G_\delta\ast(a(u,1)-d(w,u))\right)\,\d\hd^2\\
&=\int_\Gamma\left(1+\frac2\alpha\right)u_\delta^{-\frac2\alpha-2}\nabla_\Gamma u_\delta\cdot\left(u_\delta w-\varepsilon\nabla_\Gamma u_\delta\right)+u_\delta^{-\frac2\alpha-1}G_\delta\ast(d(w,u)-a(u,1))\,\d\hd^2\,.
\end{align*}
Now we can find some $C_\delta>0$ sufficiently large such that $d(w,u)-a(u,1)\leq C_\delta u$ for all times $t\in(\delta,\hat t-\delta)$, since $u$ is continuous on $[\delta,\hat t-\delta]$ and $w$ is bounded.
Thus we can estimate $G_\delta\ast(d(w,u)-a(u,1))\leq C_\delta u_\delta$ on $(2\delta,t_\delta]$.
Additionally using Young's inequality we arrive at
\begin{align*}
\frac\alpha2\frac\d{\d t}\int_\Gamma u_\delta^{-\frac2\alpha}\,\d\hd^2
&\leq\int_\Gamma\left(1+\frac2\alpha\right)u_\delta^{-\frac2\alpha-2}\left(u_\delta w\nabla_\Gamma u_\delta-\varepsilon|\nabla_\Gamma u_\delta|^2\right)+C_\delta u_\delta^{-\frac2\alpha}\,\d\hd^2\\
&\leq\int_\Gamma\left(1+\frac2\alpha\right)u_\delta^{-\frac2\alpha-2}\left(|u_\delta|^2\frac{\|w\|_{L^\infty(\Gamma)}^2}{4\varepsilon}+\varepsilon|\nabla_\Gamma u_\delta|^2-\varepsilon|\nabla_\Gamma u_\delta|^2\right)+C_\delta u_\delta^{-\frac2\alpha}\,\d\hd^2\\
&=\left(C_\delta+\frac{\|w\|_{L^\infty(\Gamma)}^2}{4\varepsilon}\left(1+\frac2\alpha\right)\right)\int_\Gamma u_\delta^{-\frac2\alpha}\,\d\hd^2\,.
\end{align*}
The above implies boundedness of $\int_\Omega u_\delta^{-2/\alpha}\,\d\hd^2$ 
for all times $t\in[2\delta,t_\delta]$,
thereby contradicting the blowup of the integral at $t_\delta$.
Consequently, $u$ does not change sign in $(0,\hat t)$.

The next a priori estimate tests equation \eqref{Eq:ProblemAbstract} with $u(t,\cdot)^{q-1}$, only again we mollify $u$ in time.
In more detail, we have
\begin{align*}
\frac1q\frac\d{\d t}\|u_\delta\|_{L^q(\Gamma)}^q
&=\int_\Gamma u_\delta^{q-1}\,\partial_t u_\delta\,\d\hd^2\\
&=\int_\Gamma u_\delta^{q-1}\left(\dive_\Gamma(\varepsilon\Delta_\Gamma u_\delta-u_\delta w)+G_\delta\ast(a(u,1)-d(w,u))\right)\,\d\hd^2\,.
\end{align*}
This time we estimate $G_\delta\ast(a(u,1)-d(w,u))\leq G_\delta\ast(C_3-C_2u^\zeta/\zeta)\leq C_3-C_2u_\delta^\zeta/\zeta$ by Jensen's inequality.
Again using Young's inequality we arrive at
\begin{align*}
\frac1q\frac\d{\d t}\|u_\delta\|_{L^q(\Gamma)}^q
&\leq\int_\Gamma u_\delta^{q-1}\dive_\Gamma(\varepsilon\Delta_\Gamma u_\delta-u_\delta w)+C_3u_\delta^{q-1}-\frac{C_2}\zeta u_\delta^{q-1+\zeta}\,\d\hd^2\\
&\leq\int_\Gamma(q-1)u_\delta^{q-2}\left(|u_\delta|^2\frac{\|w\|_{L^\infty(\Gamma)}^2}{4\varepsilon}\right)+C_3u_\delta^{q-1}-\frac{C_2}\zeta u_\delta^{q-1+\zeta}\,\d\hd^2\\
&=(q-1)\frac{\|w\|_{L^\infty(\Gamma)}^2}{4\varepsilon}\|u_\delta\|_{L^q(\Gamma)}^q+C_3\|u_\delta\|_{L^{q-1}(\Gamma)}^{q-1}-\frac{C_2}\zeta\|u_\delta\|_{L^{q-1+\zeta}(\Gamma)}^{q-1+\zeta}\\
&\leq c_1\|u_\delta\|_{L^q(\Gamma)}^q+c_2\|u_\delta\|_{L^q(\Gamma)}^{q-1}-c_3\|u_\delta\|_{L^q(\Gamma)}^{q+\zeta-1}
\end{align*}
for some positive constants $c_1,c_2,c_3$.
The right-hand side is concave in $\|u_\delta\|_{L^q(\Gamma)}^q$ and changes sign from positive to negative at some $M^q>0$.
Thus by linearizing the right-hand side in $\|u_\delta\|_{L^q(\Gamma)}^q$ about $M^q$ we can estimate
\begin{equation*}
\frac\d{\d t}\|u_\delta\|_{L^q(\Gamma)}^q
\leq-c_4\left(\|u_\delta\|_{L^q(\Gamma)}^q-M^q\right)
\end{equation*}
for some $c_4>0$.
Setting $e=\|u_\delta\|_{L^q(\Gamma)}^q-\max\{\|u_\delta(\delta,\cdot)\|_{L^q(\Gamma)}^q,M^q\}$, we thus arrive at $\frac\d{\d t}e\leq-c_4e$ with $e(\delta)\leq0$,
which by Gr\"onwall's inequality implies $e(t)\leq0$ for all $t\in[\delta,\hat t-\delta]$.
Letting now $\delta\to0$ we arrive at the desired estimate for $\|u\|_{L^q(\Gamma)}$.
\end{proof}

A direct consequence is the long-term existence of solutions to the equation.

\begin{corollary}[Long-term existence]
Let $\Gamma$, $w$, $q$, $\alpha$, and $u_0$ as in \cref{thm:globalBound},
then for any $T>0$ there exists a unique nonnegative weak solution $u\in C^0([0,T]; L^q(\Omega))\cap C^0((0,T]; C^{0,\alpha}(\Gamma))$ to \eqref{eqn:Ezrin}-\eqref{eqn:EzrinBC}
with $\|u(t,\cdot)\|_{L^q(\Gamma)}$ bounded uniformly in $t$.
\end{corollary}
\begin{proof}
The short-term solution from \cref{Thm:ShortTermExistence} exists up to some time $\hat t$ depending on $\|u_0\|_{L^q(\Gamma)}$.
By \cref{thm:globalBound} this solution is nonnegative and satisfies a uniform $L^q$-bound.
Thus, by again appealing to \cref{Thm:ShortTermExistence} for the solution of \eqref{Eq:ProblemAbstract} with initial condition $u(\hat t,\cdot)$
we can continue the solution to time $2\hat t$.
Again, \cref{thm:globalBound} yields the nonnegativity and the same $L^q$-bound.
After a finite number of repetitions of this argument we have continued the solution up to time $T$.
\end{proof}

As already mentioned previously, our system \eqref{eqn:Ezrin}-\eqref{eqn:EzrinBC} represents an advective Allen--Cahn type equation.
Such a system has already been analysed in \cite{Ref:LiuBertozziKolokolnikov2012}, however, for a simpler nonlinearity and under much higher regularity conditions.
In particular, the velocity field $w$ of \cite{Ref:LiuBertozziKolokolnikov2012} is differentiable,
and the domain is smooth enough (in their case it is an open set rather than an embedded lower-dimensional manifold) to allow classical solutions.
In more detail, the authors use the Galerkin method to show existence of solutions to the advective Cahn--Hilliard and Allen--Cahn equation with an odd polynomial as nonlinearity.
The odd polynomial allows a simpler a priori estimate than in our case where the nonlinearity behaves qualitatively different on the positive and the negative real line,
which is why in our analysis we showed nonnegativity of the solution before an $L^q$ estimate.
The authors also consider the nonlocal advective Allen--Cahn equation with mass conservation.
Here they apply the same semigroup approach as we do, also putting the advective term into the nonlinearity.
Again, the analysis is simplified by the higher regularity of domain, initial value, and velocity field.
Global boundedness and nonnegativity of the solution then follow from a maximum principle exploiting the smoothness of the solution and the velocity field.
(In fact, the authors use an Allen--Cahn well at $0$ and at $1$ and claim that the solution $u$ thus stays within $[0,1]$.
However, this is only true for velocity fields with low enough compression.
In general the argument can only provide a bound depending on the size of $\dive w$.)
In contrast, our argument makes use of as little regularity requirements as possible
(note that with little modification it even works for $w\in L^r$ with $r<\infty$ large enough);
in particular, function spaces with more than one weak derivative do not even make sense on our domain $\Gamma$.

The authors of \cite{Ref:LiuBertozziKolokolnikov2012} also study (in one spatial dimension) how droplets (regions with high values of $u$) are affected by the advection term.
In the smooth case they show that droplets cannot be broken up by an expansive flow as long as the expansion rate does not increase towards the droplet centre.
However, for general flow fields droplets can indeed be broken into smaller ones as their numerical experiments show.
In our setting we expect the flow field $w$ to be strongly expansive at the cell front and contractive in the back, thus assisting in the accumulation of Ezrin at the back of the cell.

\subsection{Phase separation}
Without the transport term $\dive_\Gamma(wu)$, \eqref{eqn:Ezrin} is a classical Allen--Cahn type equation, that is, the $L^2$-gradient flow of the energy
\begin{multline*}
E[u]=\int_\Gamma\frac\varepsilon2|\nabla u|^2+W_{w,1}(u)\,\d x\\
\text{with }
W_{w,1}(u)=\begin{cases}
\tfrac{C_1|w|+C_2}2(u^2-1)-C_3\left(\tfrac{u^{\alpha+1}-1}{\alpha+1}-\tfrac{u^{\alpha+2}-1}{\alpha+2}\right)&\text{if }u\leq1,\\
\tfrac{C_1|w|+C_2}\zeta\left(\tfrac{u^{\zeta+1}-1}{\zeta+1}+(\zeta-1)(u-1)\right)&\text{else.}
\end{cases}
\end{multline*}
It is well-known that this gradient flow results in a separation of the domain into different phases (a phase rich in Ezrin and one without) which correspond to the stable steady states of the corresponding reaction equation
\begin{equation}\label{eqn:reaction}
\partial_tu=-W_{w,1}'(u)=a(u,1)-d(w,u)=\begin{cases}C_3u^\alpha(1-u)-(C_1|w|+C_2)u&\text{if }u\leq1,\\-(C_1|w|+C_2)\frac{u^\zeta-1}\zeta+1&\text{else.}\end{cases}
\end{equation}
In this paragraph we briefly describe the involved phases.

\begin{proposition}[Ezrin phases]
In the case $\alpha=1$, a transcritical bifurcation happens at $|w|=\bar w=\frac{C_3-C_2}{C_1}$:
\begin{itemize}
\item If $|w|\geq\bar w$, \eqref{eqn:reaction} has the only nonnegative steady state $u=0$, which is stable.
\item If $|w|<\bar w$, \eqref{eqn:reaction} has two nonnegative steady states, an instable one at $u=0$ and a stable one at $u=\frac{C_3-C_2-C_1|w|}{C_3}$.
\end{itemize}
In the case $\alpha>1$, a saddle node bifurcation occurs at $|w|=\bar w=\frac{C_3}{C_1}\frac{(1-1/\alpha)^{\alpha-1}}\alpha-\frac{C_2}{C_1}$:
\begin{itemize}
\item If $|w|\geq\bar w$, \eqref{eqn:reaction} has the only nonnegative steady state $u=0$, which is stable.
\item If $|w|<\bar w$, \eqref{eqn:reaction} has exactly two stable nonnegative steady states, one at $u=0$ and one in $[\frac{\alpha-1}\alpha,1-\frac{C_2}{C_3}]$.
\end{itemize}
\end{proposition}
\begin{proof}
The case of $\alpha=1$ follows directly from noting that for $|w|\geq\bar w$ the flow field $-W_{w,1}'$ is monotonically decreasing and has a single zero in $0$,
while for $|w|<\bar w$ the two zeros of $-W_{w,1}$ are given by $u=0$ (around which $-W_{w,1}'$ is increasing) and $u=\frac{C_3-C_2-C_1|w|}{C_3}$ (around which $-W_{w,1}'$ is decreasing).

In the case $\alpha>1$ we have $-W_{w,1}'(0)=0$ and $-W_{w,1}''(0)<0$ so that $u=0$ is always a stable steady state.
From the shape of $-W_{w,1}'$ (cf.\ \cref{fig:doubleWell2}) it is then obvious that, depending on $|w|$, this is either the only zero of $W_{w,1}'$
or there are two more (thus the smaller one must be instable and the larger stable).
That the larger one lies within $[\frac{\alpha-1}\alpha,1-\frac{C_2}{C_3}]$ can easily be shown by checking $-W_{w,1}'(1-\frac{C_2}{C_3})<0$ and $-W_{w,1}'(\frac{\alpha-1}\alpha)>0$ for $|w|<\bar w$.
That $w=\bar w$ is critical can be seen from noticing that $-W_{\bar w,1}'$ has a maximum in $u=\bar u=1-\frac1\alpha$, at which $-W_{\bar w,1}'(\bar u)=0$.
\end{proof}

Thus, the expected final configuration of the evolving system is as follows:
In vicinity of the actin brushes the flow $w$ is strong so that there the phase $u=0$ is the only stable state of the reaction part in \eqref{eqn:Ezrin}.
Thus, all throughout that region $u$ quickly decreases to zero, and the transport term $\dive_\Gamma(wu)$ has little influence in that region due to $u\approx0$ so that after a while the equation indeed almost behaves like the gradient flow of $E$.
Further away from the actin brushes, where the flow $w$ becomes sufficiently weak,
the Ezrin-rich phase with $u>0$ suddenly becomes stable.
Here we have to distinguish between the behaviour for $\alpha=1$ and $\alpha>1$.
\begin{itemize}
\item
For $\alpha=1$, the stable state of the Ezrin-reaction is directly coupled to the flow by $u=\max\{0,\frac{C_3-C_2-C_1|w|}{C_3}\}$.
In particular, towards the back of the cell the flow will be so weak that $u$ will approximately take the value $u\approx1-\frac{C_2}{C_3}$.
Again the transport term $\dive_\Gamma(wu)$ becomes negligible in that region due to the smallness of $w$ and $\dive_\Gamma w$ and the constancy of $u$
so that after a while the equation roughly behaves like a gradient flow of $E$.
Now if the transition from high boundary velocity $|w|$ at the cell front to low velocity in the back is gradual,
then also $u$ will change gradually along the cell membrane from almost $0$ to almost $1-\frac{C_2}{C_3}$.
If on the other hand the flow velocity changes rather abruptly (for instance in the wake of the nucleus),
then the transition width from Ezrin-low to Ezrin-rich phase is governed by the diffusion $\epsilon\Delta_\Gamma u$.
In that case it is known since the work of Modica and Mortola \cite{Ref:ModicaMortola1977} that this width roughly behaves like
\begin{equation*}
\sqrt{\tfrac{\varepsilon}{W_{w,1}(u)-W_{w,1}(1-C_2/C_3)}}(1-\tfrac{C_2}{C_3})\sim\sqrt{\tfrac\varepsilon{C_3}}(1-\tfrac{C_2}{C_3})^{-3/2}
\end{equation*}
for $u$ between the two phases.
\item
For $\alpha>1$ there will be no gradual change between the Ezrin-rich and Ezrin-low phase since there is a concentration gap of $\frac{\alpha-1}\alpha$ between both phases so that $u$ will more or less jump from $u\approx0$ to $u>\frac{\alpha-1}\alpha$.
As in the case $\alpha=1$, this jump will be regularized by the diffusion and happen across a width of $\sqrt{\tfrac{\varepsilon}{W_{w,1}(u)-W_{w,1}(1-1/\alpha)}}(1-\tfrac1\alpha)$.
\end{itemize}


\section{Comparison and discussion of numerical and biological experiments} \label{sec:Experiments}

Verification of the model is nontrivial since inside a germ cell the model parameters can only be influenced indirectly, cannot be properly quantified, and typically also govern other processes that can completely alter the cell behaviour.
In this section we vary parameters in numerical simulations of the model and compare the results to biological experiments in which the parameters are subjected to qualitatively similar changes.
For simplicity (and since a quantitative comparison to biological experiments is ruled out anyway) we discretize and simulate the equations in two rather than three spatial dimensions.
Again we will drop the hats on all variables.

\subsection{Finite Element Discretization}\label{sec:Discretization}

The weak form of \eqref{eqn:Darcy1}-\eqref{eqn:DarcyBC} (we will replace \eqref{eqn:Darcy2} by the equivalent \eqref{eqn:DarcyP}) as well as \eqref{eqn:Ezrin} consists in finding $p\in H^1(\Omega)$, $w\in L^2(\Omega;\R^2)$ and $u\in C([0,T];H^1(\Gamma))\cap C^1([0,T];H^{-1}(\Gamma))$ such that
\begin{align*}
\int_\Omega \nabla p\cdot\nabla\psi \,\d x &= \int_\Omega f\cdot\nabla\psi \,\d x
&&\forall\psi\in H^1(\Omega)\,,\\
\int_\Omega w\cdot q \,\d x &= \int_\Omega p\,\dive q + f\cdot q \,\d x
&&\forall q\in L^2(\Omega)\,,\\
\int_\Gamma \partial_tu \, \varphi \,\d x &= \int_\Gamma uw\cdot\nabla_\Gamma\varphi - \varepsilon\nabla_\Gamma u\cdot\nabla_\Gamma\varphi - d(w,u)\varphi + a(u,1)\varphi \,\d x
&&\forall \varphi\in H^1(\Gamma),t\in[0,T]\,.
\end{align*}
Note that here we use the standard spaces for the Ezrin equation on $\Gamma$ rather than the minimum regularity spaces from the previous section, since in our numerical experiments the domains and thus also solutions will be smooth.
We will use semi-implicit finite differences in time, using a step length $\Delta t = \frac{T}{M}$ for some $M\in\N$, and finite elements in space.
In order to allow a potential generalization of the model from Darcy to viscous Darcy or Stokes--Brinkmann flow we use Taylor--Hood elements.
Thus, given a triangulation $\mathcal T_h$ of a polygonal approximation $\Omega_h$ of our domain and a discretization $\mathcal S_h$ of $\Gamma_h$ into line segments we use the finite element spaces
\begin{align*}
\mathbf W_h&=\left\{w\in C(\Omega_h;\R^2)\,\middle|\,w|_T\in\mathcal P_2\text{ for all }T\in\mathcal T_h\right\}\,,\\
\mathbf P_h&=\left\{p\in C(\Omega_h)\,\middle|\,p|_T\in\mathcal P_1\text{ for all }T\in\mathcal T_h\right\}\,,\\
\mathbf U_h&=\left\{w\in C(\Gamma_h)\,\middle|\,w|_T\in\mathcal P_2\text{ for all }S\in\mathcal S_h\right\}\,,
\end{align*}
where $\mathcal P_k$ is the space of polynomials of degree no larger than $k$.
In summary, the discretized version of the equations is to find $p_h\in\mathbf P_h$, $w_h\in\mathbf W_h$ and $u_h^0,\ldots,u_h^M\in\mathbf U_h$ such that
\begin{align*}
\int_{\Omega_h} \nabla p_h^{k+1}\cdot\nabla\psi_h\, \d x &= \int_{\Omega_h} f\cdot\nabla\psi_h \, \d x
&&\forall\psi_h\in\mathbf P_h\,,\\
\int_{\Omega_h} w_h\cdot q_h \, \d x &= \int_{\Omega_h} p_h\dive_\Gamma\, q_h + f\cdot q_h \, \d x
&&\forall q_h\in\mathbf W_h\,,\\
\int_{\Gamma_h} \frac{u_h^{k+1}-u_h^k}{\Delta t}\varphi_h \, \d x &= \int_{\Gamma_h} u_h^{k+1}w_h\cdot\nabla_{\Gamma_h}\varphi_h - \varepsilon\nabla_{\Gamma_h} u_h^{k+1}\cdot\nabla_{\Gamma_h}\varphi_h\\
&\qquad- d(w_h,u_h^{k})\varphi_h + a(u_h^{k},1)\varphi_h \, \d x
&&\forall\varphi_h\in\mathbf U_h,k=0,\ldots,M-1\,.
\end{align*} 
In principle one might let $f$ also depend on time in which case $p_h$ and $w_h$ would become the corresponding time-dependent quasistationary solutions.
Since a higher spatial resolution for $u$ was desired, we chose to work with a finer polygonal approximation $\Gamma_h$ of $\Gamma$ than $\partial\Omega_h$ and used nearest neighbour interpolation to approximate $w|_{\Gamma_h}$ by $w|_{\partial\Omega_h}$.
Denote by $M,S,B$ the mass, stiffness and mixed mass-stiffness matrix and by $M^\omega$ the weighted mass matrix with weight function $\omega$.
Using capital letters for the vectors of degrees of freedom in $p_h$, $w_h$, and $u_h^k$, the discrete system can be rewritten as
\begin{align*}
SP_h &= B^TF\,, \\
MW_h &= -BP_h+MF\,, \\
\left(\frac{1}{\Delta t}M-A^{w_h}+\varepsilon S\right)U_h^{k+1} &= \frac{1}{\Delta t}MU_h^k - M^{d(w_h,u_h^k)}\mathbf{1} + M^{a(u_h^k,1)}\mathbf{1}\,,
\end{align*}
where $\mathbf 1$ denotes the vector representing the constant function $1$.
Here, the matrix $A^\omega$ is defined as $A_{ij}^\omega = \int_{\Gamma_h} \varphi_j\omega\cdot\nabla_{\Gamma_h}\varphi_i\,\d x$ for the basis functions $\varphi_i$ of $\mathbf U_h$.
For a fixed $\varepsilon$ and small enough time step $\Delta t$, the matrix $\left( \frac{1}{\Delta t}M-A+\varepsilon S \right)$ is positive definite so that the last equation can be solved.
Likewise, $M$ is invertible and $S$ is invertible on the subspace $\{p\in\mathbf P_h\,|\,\int_{\Omega_h}p\,\d x=0\}$ so that the first two equations can be solved as well.

\subsection{Experimental study} \label{sec:ExperimentalResults}

We implemented the scheme from the previous section in MATLAB$^{\tiny{\textcircled{c}}}$
and performed simulations on an elliptical cell $\Omega$ with a circular impermeable nucleus $N\subset\Omega$.
Our chosen default parameters are provided in \cref{tab:parameters}.
In dimensional variables these correspond to the values from \cref{sec:Nondimensionalisation} as well as the following.

\bigskip

\begin{tabular}{ll}
width and length of the cell&12\,$\mu$m and 18\,$\mu$m\\
diameter of the nucleus&6\,$\mu$m\\
cytoplasmic Ezrin concentration&$10^{-18}$mol/$\mu$m$^3$\\
membraneous Ezrin concentration&$10^{-22}$mol/$\mu$m$^2$\\
$(\gamma,\beta_1,\beta_2)$&$(3\cdot 10^{38},3.3,0.0007)$
\end{tabular}

\bigskip

\begin{table}
\centering
\begin{tabular}{ll}
parameter&value\\\hline
$T$&1\\
$\Omega$&$\{x\in\R^2\,|\,(x_1/1.2)^2+(x_2/0.8)^2\leq1\}$\\
nucleus $N$&$\{x\in\R^2\,|\,|x-(0.2 0)|<0.4\}$\\
$|f|$&$20$\\
$\varepsilon$&$0.002$\\
$(\alpha,\zeta)$&$(1,2)$\\
$(C_1,C_2,C_3)$&$(50,0.1,5)$
\end{tabular}
\caption{Default parameter values for the numerical simulation of the cytoplasmic flow and Ezrin concentration in a primordial germ cell.}
\label{tab:parameters}
\end{table}

The force $f$ which represents the actin brushes is a smoothed point force, $f=\vec f G$ with a smooth nonnegative, compactly supported kernel $G$ of mass $1$, located close to the cell front and pointing forward.
Its strength is chosen as $|\vec f|=20$ ($20$ pN/$\mu$m$^2$ in physical dimensions) so as to achieve an average cytoplasmic velocity corresponding to $0.1\,\mu$m/s, which is the value estimated from a few biological microscopy videos.
The resulting simulated intracellular flow is shown in \cref{fig:IntracellularFlowSimulation1}. 

\begin{figure}
	\centering
	\begin{tikzpicture}
		\node at (0,0) {\includegraphics[width=1\textwidth]{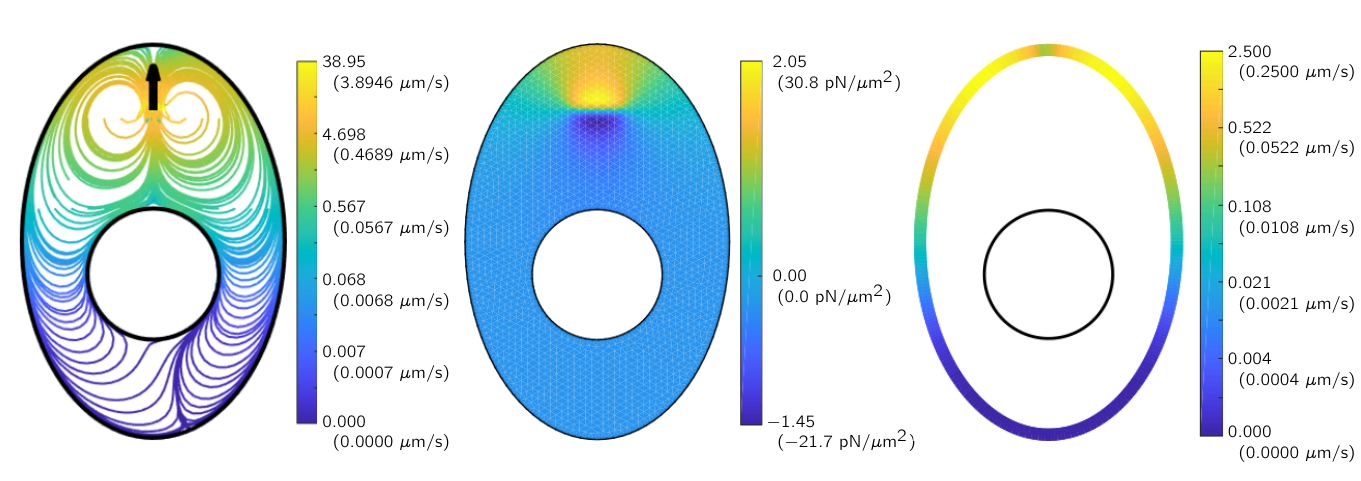}};
		\put(-154,57) {\small{$\mathrm{front}$}};
		\put(-154,-64) {\small{$\mathrm{back}$}};
		\put(-33,57) {\small{$\mathrm{front}$}};
		\put(-33,-64) {\small{$\mathrm{back}$}};
		\put(91,57) {\small{$\mathrm{front}$}};
		\put(91,-64) {\small{$\mathrm{back}$}};
	\end{tikzpicture}
\caption{Intracellular flow resulting from \eqref{eqn:Darcy1}-\eqref{eqn:DarcyBC}.
Left: Streamlines, colour-coded according to velocity (logarithmic scale; location and direction of the actin brush-induced forces are indicated by the black arrow).
Middle: Intracellular pressure distribution.
Right: Magnitude of boundary velocity (logarithmic scale).
}
\label{fig:IntracellularFlowSimulation1}
\end{figure}
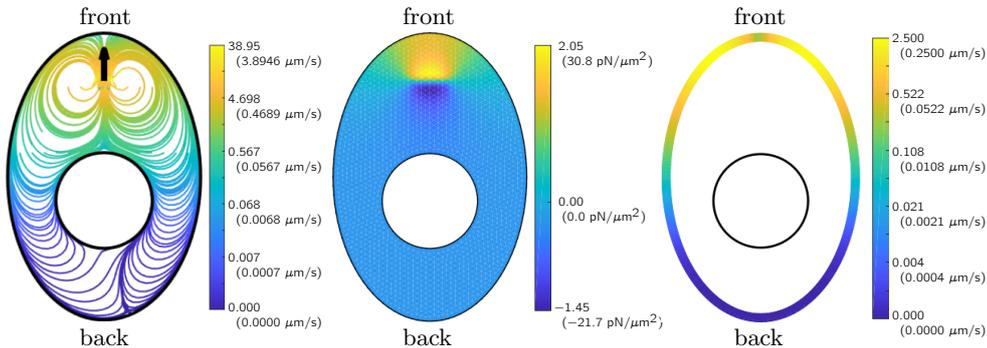

As for the numerical simulation of \eqref{eqn:Ezrin}-\eqref{eqn:EzrinBC},
we typically started from a random Ezrin distribution on $\Gamma$ at time $0$
(that is, $u(0,x)$ for each grid point $x$ was sampled according to the uniform distribution on $[0,1]$).
A stationary state was then typically reached towards the end of the simulated time interval (see \cref{fig:timeEvolution}).
Note that even though exactly at the cell front there is no shear stress, the diffusion is strong enough to remove Ezrin there as well.

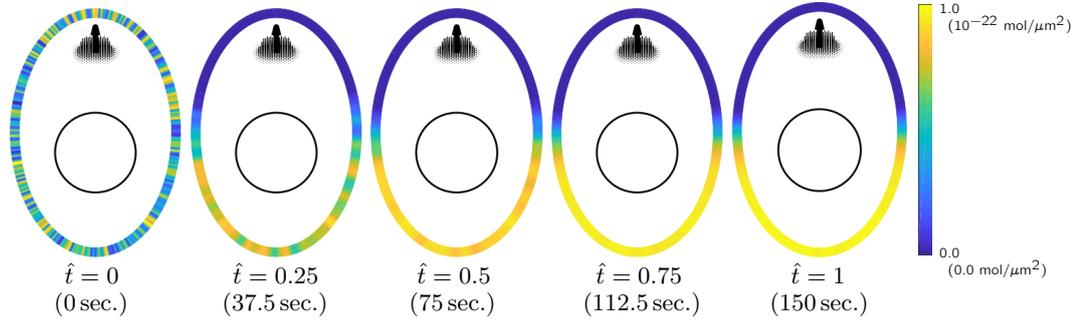
\begin{figure}
	\centering
	\begin{tikzpicture}
	\node at (0,0) {\includegraphics[width=1.1\textwidth]{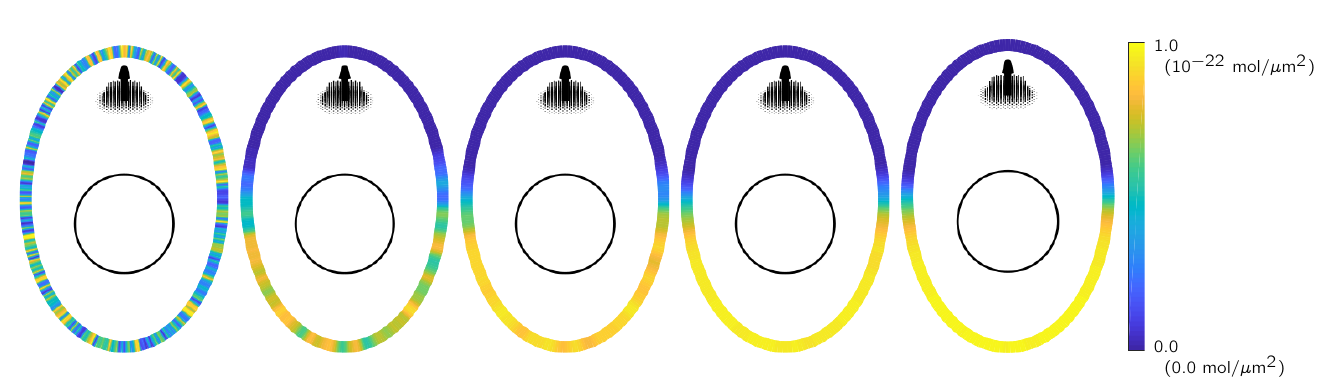}};
	\put(-179,-58) {\small{$\hat{t}=0$}} 
	\put(-116,-58) {\small{$\hat{t}=0.25$}} 
	\put(-47,-58) {\small{$\hat{t}=0.5$}} 
	\put(20,-58) {\small{$\hat{t}=0.75$}} 
	\put(93,-58) {\small{$\hat{t}=1$}} 
	\put(-183,-68) {\small{$(0 \, \mathrm{sec.})$}}
	\put(-120,-68) {\small{$(37.5 \, \mathrm{sec.})$}}
	\put(-50,-68) {\small{$(75 \, \mathrm{sec.})$}}
	\put(15,-68) {\small{$(112.5 \, \mathrm{sec.})$}}
	\put(85,-68) {\small{$(150 \, \mathrm{sec.})$}}
	\end{tikzpicture}
\caption{Time evolution of the active Ezrin distribution during actin-brush-induced polarization. Towards the end of the simulation a stationary state is reached. Location and direction of the actin brush-induced forces are indicated in black.
}
\label{fig:timeEvolution}
\end{figure}

In the remainder of this section we try to compare the model behaviour to the observed biological cell behaviour.
To this end we varied those model parameters in our simulations which we expect to strongly determine the resulting Ezrin concentration and which could be influenced also experimentally.
Subsequently we tried to reproduce these model changes in biological experiments so as to compare the biological result to our simulation.

To prepare the experiments, the zebrafish embryos were injected at 1-cell stage into the yolk with 1nl of the sense mRNA and translation blocking morpholino antisense oligonucleotide against chemoattractant Cxcl12a (MOs; GeneTools).
Messenger RNA was synthesized using the mMessageMachine kit (Ambion).
To express proteins preferentially in germ cells the corresponding coding region was cloned upstream of the 3’UTR of nanos3 gene
\cite[K\"oprunner, Thisse et al., 2001]{Ref:KoeprunnerEtAl2001}. 
Embryos were incubated at 25 degrees prior to imaging.
The representative biological images were acquired with the VisiView software using a widefield fluorescence microscope (Carl Zeiss Axio imager Z1 with a Retiga R6 camera (experiments 1-3)). 
Confocal microscopy imaging was done using a Carl Zeiss LSM 710 microscope and time lapse movies were taken using the Zen software (experiment 4).
Photoactivation was performed using a Carl Zeiss LSM 710 confocal microscope in the bleaching menu of the Zen software.
At the beginning, $5$ frames (every $7.75$ seconds) were imaged before the first round of photoactivation with a $456$nm laser ($120-240$ iterations, 100\,\% laser power). 
Photoactivation was performed near the cell boundaries, opposing existing actin brushes in a circular region of interest (diameter of $5\mu$m) and repeated every $4$ frames.

In the following we describe the single experiments; the corresponding numerical and biological results are shown in \cref{fig:Experiments}.
\begin{enumerate}
\item
\emph{Changes in total Ezrin concentration.}
The total Ezrin concentration enters in the adsorption strength $C_3$ at the membrane. Its increase will thus increase the amount of active Ezrin on the membrane.
Biologically, an Ezrin overexpression (experiment 1(a)) was achieved by injection of Ypet.Ezrin.nos3'UTR RNA construct ($250$\,pg). A comparable decrease of the Ezrin concentration is more difficult to achieve, thus this was only performed numerically.
\Cref{fig:Experiments}, 1(a) shows the numerical result for an increased value of the default $C_3$ by factor $5$ alongside with a microscopy image of the biological experiment.
While active Ezrin still accumulates in the back, its concentration is elevated on a much larger portion of the cell membrane than usual.
\Cref{fig:Experiments}, 1(b) shows the numerical result for a decreased value of $C_3$ by factor $1/10$, resulting in an almost vanishing Ezrin concentration on the membrane.
\item
\emph{Changes in actin activity.}
Numerically, an increased (or decreased) actin activity can be achieved by an in- (or de-)creased force strength $|\vec f|$, resulting in a higher (or lower) intracellular velocity (in experiment 2(a) we use a $5$-fold increase, in experiment 2(b) a decrease by the factor $1/10$). Biologically, the cell contractility was increased (experiment 2(a)) by injection of CA-RhoA.nos3'UTR RNA construct ($15$\,pg) \cite[Paterson, Self et al., 1990]{Ref:Paterson1990} and decreased (experiment 2(b)) by injecting DN-ROCK.nos3'UTR RNA construct ($150$\,pg) \cite[Blaser, Reichmann-Fried et al., 2006]{Ref:BlaserEtAl2006}. 
Both numerical and experimental results are displayed in \cref{fig:Experiments}, 2(a)-(b).
As the microscopy image in 2(a) shows, an increased contractility quickly leads to the formation of a huge bleb, consistent with the complete Ezrin depletion over a large part of the cell membrane observed in the simulation. In case of a decreased actin activity, both the experimental and numerical results in 2(b) show a high active Ezrin concentration over the major part of the membrane.  
\item
\emph{Changes in cell shape and nucleus position.}
The intracellular flow is strongly influenced by the cell shape as well as the position of the nucleus, which shields the back of the cell from the flow. We tested the effect of the cell shape on the equilibrium configuration of active Ezrin by decreasing the semi-minor and increasing the semi-major axis of the cell to $5/18$ and $9/5$, while keeping the cell volume constant. Additionally, we changed the position of the nucleus by shifting it closer to one side of the cell.
In order to observe the same behaviour in the microscopy images, the zebrafish embryos were injected with a moderate amount of Ypet.Ezrin.nos3'UTR RNA construct ($80$\,pg). The cells were followed by time lapse imaging, and the position of the nucleus was examined. 
The results are displayed in \cref{fig:Experiments}, 3(a)-(c).
While in the simulation results for the long and thin cell, the flow was partially blocked by the nucleus, the high density of active Ezrin in the microscopy image seems to be restricted to the back side of the cell. Reasons for the different behaviour could be the restriction to a two-dimensional region in the simulations or the different proportions of the nucleus and actual cell size. The simulation results for experiments 3(b) and 3(c) however fit the biological images. 
\item
\emph{Counteracting flows.}
Numerically it is straightforward to add a second layer of actin brushes that counteracts the flow of the first actin brushes.
The result is a flow-induced removal of active Ezrin in the cell front as well as the back so that active Ezrin only remains at the cell sides (\cref{fig:Experiments}, 4). 
Biologically, the embryos were injected with Ypet.Ezrin.nos3'UTR ($80$\,pg) and photoactivatable Rac1 construct ($150$\,pg, \cite[Wu, Frey et al., 2009]{Ref:WuEtAl2009}), which is a version of small GTPase Rac1 that becomes active when excited with the $456$\,nm light, promoting actin polymerization. Thereby, it is possible to generate actin brushes in a second location within a cell. In the mathematical simulations as well as in the corresponding microscopy images, active Ezrin accumulates at the cell sides, in case of the biological experiment it remains restricted to one side, presumably due to the slight asymmetry of the cell. 
\end{enumerate}


\begin{figure}
	\centering
	\includegraphics[width=1.1\textwidth]{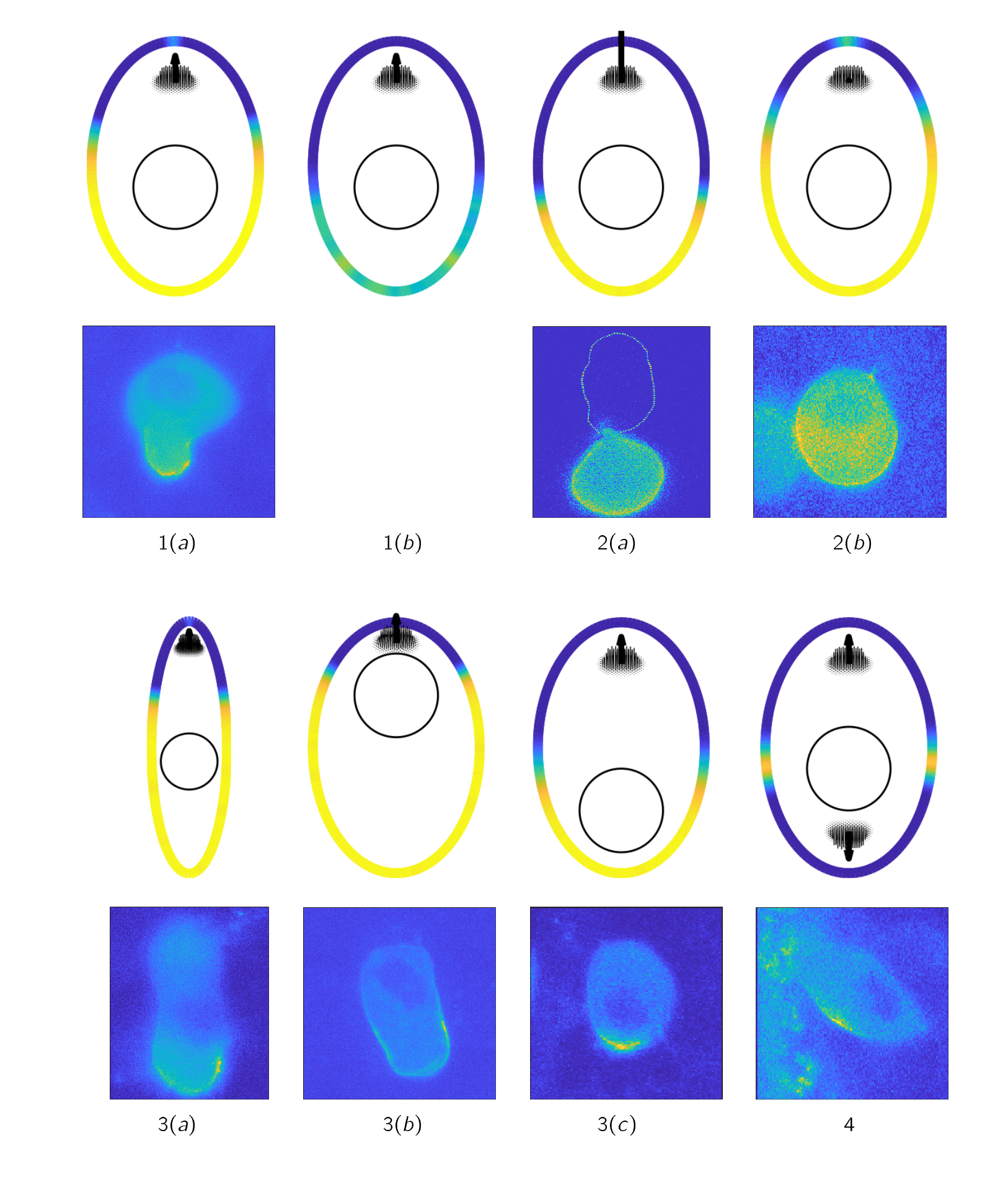}
	\caption{All experiments described in section \ref{sec:ExperimentalResults} next to each other. The top shows the numerical simulation result and the bottom a corresponding microscopy image.}
	\label{fig:Experiments}
\end{figure}

Overall we find good qualitative agreement between numerical and experimental results,
supporting the hypothesis that Ezrin destabilization at the cell membrane by an actin brush-induced intracellular flow may be the mechanism finally leading to bleb formation.

\section*{Acknowledgement}

The work was supported by the Alfried Krupp Prize for Young University Teachers awarded by the Alfried Krupp von Bohlen und Halbach-Stiftung, by the Deutsche Forschungsgemeinschaft (DFG, German Research Foundation) via Germany's Excellence Strategy through the Clusters of Excellence ``Cells-in-Motion'' (EXC 1003, particularly project PP-2017-07) and ``Mathematics M\"unster: Dynamics -- Geometry -- Structure'' (EXC 2044) at the University of M\"unster as well as by the European Research Council (ERC, CellMig) and the German Research Foundation (DFG RA 863/11-1). 

\bibliographystyle{plain}
\bibliography{mybibliography}

\end{document}